\documentclass[11pt]{article}

\usepackage{epsfig,epsf,fancybox}
\usepackage{amsmath}
\usepackage{mathrsfs}
\usepackage{amssymb}
\usepackage{graphicx}
\usepackage{color}
\usepackage[linktocpage,pagebackref,colorlinks,linkcolor=blue,anchorcolor=blue,citecolor=blue,urlcolor=blue,hypertexnames=false]{hyperref}
\usepackage{boxedminipage}
\usepackage{stmaryrd}
\usepackage{multirow}
\usepackage{booktabs}
\usepackage{accents}
\usepackage{cite}
\usepackage{float}
\usepackage[ruled,vlined,linesnumbered]{algorithm2e}

\newtheorem{theorem}{Theorem}[section]

\newtheorem{lemma}[theorem]{Lemma}
\newtheorem{proposition}[theorem]{Proposition}

\,
\,

\newcommand{\be}{\begin{equation}}
\newcommand{\ee}{\end{equation}}
\newcommand{\ba}{\begin{array}}
\newcommand{\ea}{\end{array}}
\newcommand{\bpm}{\begin{pmatrix}}
\newcommand{\epm}{\end{pmatrix}}

\newcommand{\bea}{\begin{eqnarray}}
\newcommand{\eea}{\end{eqnarray}}
\newcommand{\beaa}{\begin{eqnarray*}}
\newcommand{\eeaa}{\end{eqnarray*}}

\newcommand{\bal}{\begin{align}}
\newcommand{\eal}{\end{align}}
\newcommand{\baln}{\begin{align*}}
\newcommand{\ealn}{\end{align*}}

\usepackage[normalem]{ulem} 

\newcommand{\vu}{u}
\newcommand{\vv}{v}

\newcommand{\vareps}{\varepsilon}

\newcommand{\vS}{S}

\newcommand{\cA}{{\mathcal{A}}}

\newcommand{\cK}{{\mathcal{K}}}
\newcommand{\cL}{{\mathcal{L}}}

\newcommand{\cN}{{\mathcal{N}}}

\newcommand{\cP}{{\mathcal{P}}}

\newcommand{\cU}{{\mathcal{U}}}

\newcommand{\cX}{{\mathcal{X}}}

\newcommand{\RR}{\mathbb{R}} 

\newcommand{\st}{\mbox{ s.t. }}

\DeclareMathOperator*{\argmin}{arg\,min} 

\newcommand{\bc}{\begin{center}}
\newcommand{\ec}{\end{center}}

\newcommand{\bdm}{\begin{displaymath}}
\newcommand{\edm}{\end{displaymath}}

\newcommand{\beq}{\begin{equation}}
\newcommand{\eeq}{\end{equation}}

\newcommand{\bfl}{\begin{flushleft}}
\newcommand{\efl}{\end{flushleft}}

\newcommand{\bt}{\begin{tabbing}}
\newcommand{\et}{\end{tabbing}}

\newcommand{\beqn}{\begin{eqnarray}}
\newcommand{\eeqn}{\end{eqnarray}}

\newcommand{\beqs}{\begin{align*}} 
\newcommand{\eeqs}{\end{align*}}  

\newtheorem{assumption}{Assumption}

\begin{document}

\title{Hybrid Jacobian and Gauss-Seidel proximal block coordinate update methods for linearly constrained convex programming}

\author{Yangyang Xu\thanks{xuy21@rpi.edu. Department of Mathematical Sciences, Rensselaer Polytechnic Institute, Troy, NY 12180}
}

\date{}

\maketitle

\begin{abstract}
Recent years have witnessed the rapid development of block coordinate update (BCU) methods, which are particularly suitable for problems involving large-sized data and/or variables. In optimization, BCU first appears as the coordinate descent method that works well for smooth problems or those with separable nonsmooth terms and/or separable constraints. As nonseparable constraints exist, BCU can be applied under primal-dual settings.

In the literature, it has been shown that for weakly convex problems with nonseparable linear constraint, BCU with fully Gauss-Seidel updating rule may fail to converge and that with fully Jacobian rule can converge sublinearly. However, empirically the method with Jacobian update is usually slower than that with Gauss-Seidel rule. To maintain their advantages, we propose a hybrid Jacobian and Gauss-Seidel BCU method for solving linearly constrained multi-block structured convex programming, where the objective may have a nonseparable quadratic term and separable nonsmooth terms. At each primal block variable update, the method approximates the augmented Lagrangian function at an affine combination of the previous two iterates, and the affinely mixing matrix with desired nice properties can be chosen through solving a semidefinite programming. We show that the hybrid method enjoys the  theoretical convergence  guarantee as Jacobian BCU. In addition, we numerically demonstrate that the method can perform as well as Gauss-Seidel method and better than a recently proposed randomized primal-dual BCU method.

\vspace{0.5cm}

\vspace{0.5cm}


\end{abstract}

\section{Introduction}
Driven by modern applications in image processing, statistical and machine learning, block coordinate update (BCU) methods have revived in recent years. BCU methods decompose a complicated large-scale problem into easy small subproblems and tend to have low per-update complexity and low memory requirement, and they give rise to powerful ways to handle problems involving large-sized data and/or variables. These methods originate from the coordinate descent method that only applies to optimization problems with separable constraints. Under primal-dual settings, they have been developed to deal with nonseparably constrained problems.

In this paper, we consider the linearly constrained multi-block structured problem with a quadratic term in the objective:
\begin{equation}\label{eq:mb-prob}
\min_x \frac{1}{2}x^\top Qx+\sum_{i=1}^m g_i(x_i), \st A x=b,
\end{equation}
where the variable is partitioned into $m$ disjoint blocks $x=(x_1,x_2,\ldots,x_m)$, $Q$ is a positive semidefinite (PSD) matrix, and each $g_i$ is a proper closed convex and possibly non-differentiable function. Note that part of $g_i$ can be an indicator function of a convex set $\cX_i$, and thus \eqref{eq:mb-prob} can implicitly include certain separable block constraints $x_i\in\cX_i$ in addition to the nonseparable linear constraint. 

Due to its multi-block and also coordinate friendly \cite{peng2016cf} structure, we will derive a BCU method for solving \eqref{eq:mb-prob}, by performing BCU to $x_i$'s based on the augmented Lagrangian function of \eqref{eq:mb-prob}, followed by an update to the Lagrangian multiplier; see the updates in \eqref{eq:update-x} and \eqref{eq:update-lam}.

\subsection{Motivations}
This work is motivated from two aspects. First, many applications can be formulated in the form of \eqref{eq:mb-prob}. Second, although numerous optimization methods can be applied to these problems, few of them are reliable and also efficient. Hence, we need a novel algorithm that can be applied to all these applications and also has a nice theoretical convergence result.
 
\subsection*{Motivating examples}
 If $g_i$ is the indicator function of the nonnegative orthant for all $i$, then it reduces to the nonnegative linearly constrained quadratic programming (NLCQP):
\begin{equation}\label{eq:nonneg-qp}
\min_x \frac{1}{2} x^\top Q x + d^\top x, \st Ax=b,\, x\ge 0.
\end{equation}
All convex QPs can be written as NLCQP by adding slack variable and/or decomposing a free variable into positive and negative parts. As $Q$ is a huge-sized matrix, it would be beneficial to partition it into block matrices and correspondingly $x$ and $A$ into block variables and block matrices, and then apply BCU methods toward finding a solution to \eqref{eq:nonneg-qp}. 

Another example is the constrained Lasso regression problem proposed in \cite{james2012classo}: 
\begin{equation}\label{eq:classo}
\min_x \frac{1}{2}\|Ax-b\|^2 + \mu\|x\|_1, \st Cx\le d.
\end{equation}
If there is no constraint $Cx\le d$, \eqref{eq:classo} simply reduces to the Lasso regression problem \cite{tibshirani1996lasso}. Introducing a nonnegative slack variable $y$, we can write \eqref{eq:classo} into the form of \eqref{eq:mb-prob}:
\begin{equation}\label{eq:classo2}
\min_{x, y} \frac{1}{2}\|Ax-b\|^2 + \mu\|x\|_1 + \iota_+(y), \st Cx+y = d,
\end{equation}
where $\iota_+(y)$ is the indicator function of the nonnegative orthant, equaling \emph{zero} if $y$ is nonnegative and $+\infty$ otherwise. Again for large-sized $A$ or $C$, it is preferable to partition $x$ into disjoint blocks and apply BCU methods to \eqref{eq:classo2}.

There are many other examples arising in signal and image processing and machine learning such as the compressive principal component pursuit \cite{wright2013compress-PCP} (see \eqref{eq:cpcp2} below) and the regularized multiclass support vector machines \cite{xu2015admm-msvm} (see \eqref{eq:msvm} below for instance). More examples, including basis pursuit, conic programming, and the exchange problem, can be found in \cite{he2012ADM-GBS, deng2016parallel, hong2014block, sun2015convergent, li2016schur} and the references therein.

\subsection*{Reliable and efficient algorithms}
Towards a solution to \eqref{eq:mb-prob}, one may apply any traditional method, such as projected subgradient method, augmented Lagrangian method (ALM), and the interior-point method \cite{nocedal2006numopt}. However, these methods do not utilize the block structure of the problem and are not fit to very large-scale problems. To utilize the block structure, BCU methods are preferable. For unconstrained or block-constrained problems, recent works (e.g., \cite{XY_2013_multiblock, razaviyayn2013unified, XY_2017_ecd, wright2015coordinate, hong2016iteration}) have shown that BCU can be theoretically reliable and also practically efficient. Nonetheless, for problems with nonseparable linear constraints, most existing BCU methods either require strong assumptions for convergence guarantee or converge slowly; see the review in section \ref{sec:review} below. Exceptions include \cite{he2012ADM-GBS, he2017convergence} and \cite{dang2014randomized, GXZ-RPDCU2016}. However, the former two only consider separable convex problems, i.e., without the nonseparable quadratic term in \eqref{eq:mb-prob}, and the convergence result established by the latter two is stochastic rather than worse-case. In addition, numerically we notice that the randomized method in \cite{GXZ-RPDCU2016} performs not so well when the number of blocks is small. Our novel algorithm utilizes the block structure of \eqref{eq:mb-prob} and also enjoys fast worse-case convergence rate under mild conditions.

\subsection{Related works}\label{sec:review}
BCU methods in optimization first appear in \cite{Hildreth-57} as the coordinate descent (CD) method for solving quadratic programming with separable nonnegative constraints but without nonseparable equality constraint. The CD method updates one coordinate every time while all the remaining ones are fixed. It may stuck at a non-stationary point if there are nonseparable nonsmooth terms in the objective; see the example in \cite{Warga-63, shi2016primer-cd}. On solving smooth problems or those with separable nonsmooth terms, the convergence properties of the CD method have been intensively studied (e.g., \cite{Tseng-01, tseng2009_CGD, XY_2013_multiblock, razaviyayn2013unified, XY_2017_ecd, wright2015coordinate, hong2016iteration}). For the linearly constrained problem \eqref{eq:mb-prob}, the CD method can also stuck at a non-stationary point, for example, if the linear constraint is simply $x_1-x_2=0$. Hence, to directly apply BCU methods to linearly constrained problems, at least two coordinates need be updated every time; see \cite{tseng2009block-linear, necoara2014random-cd} for example.

Another way of applying BCU towards finding a solution to \eqref{eq:mb-prob} is to perform primal-dual block coordinate update (e.g., \cite{pesquet2014class, Peng_2016_AROCK, combettes2015stochastic, peng2016cf}). These methods usually first formulate the first-order optimality system of the original problem and then apply certain operator splitting methods. Assuming monotonicity of the iteratively performed operator (that corresponds to convexity of the objective), almost sure iterate sequence convergence to a solution can be shown, and with strong monotonicity assumption (that corresponds to strong convexity), linear convergence can be established.
 
In the literature, there are also plenty of works applying BCU to the augmented Lagrangian function as we did in \eqref{eq:update-x}. One popular topic is the alternating direction method of multipliers (ADMM) applied to separable multi-block structured problems, i.e., in the form of \eqref{eq:mb-prob} without the nonseparable quadratic term. Originally, ADMM was proposed for solving separable two-block problems \cite{Glowinski1975, gabay1976dual} by cyclicly updating the two block variables in a Gauss-Seidel way, followed by an update to the multiplier, and its convergence and also $O(1/t)$ sublinear rate is guaranteed by assuming merely weak convexity (e.g., \cite{monteiro2010iteration, he2012rate-DRS, LMZ2015JORSC}). While directly extended to problems with more than two blocks, ADMM may fail to converge as shown in \cite{chen2016direct} unless additional assumptions are made such as strong convexity on part of the objective (e.g., \cite{han2012note, chencai2013convergence, cai2014directstrong, LMZ2015JORSC, lin2015global, li2015convergent, davis2015-3op}), orthogonality condition on block coefficient matrices in the linear constraint (e.g., \cite{chen2016direct}), and Lipschitz differentiability of the objective and invertibility of the block matrix in the constraint about the last updated block variable (e.g., \cite{wang2015ncvx-admm}). For problems that do not satisfy these conditions, ADMM can be modified and have guaranteed convergence and even rate estimate by adding extra correction steps (e.g., \cite{he2012ADM-GBS, he2017convergence}) or by randomized block coordinate update (e.g., \cite{dang2014randomized, GXZ-RPDCU2016}) or adopting Jacobian updating rules (e.g., \cite{deng2016parallel, he2015full, li2016schur}) that essentially reduce the method to proximal ALM or two-block ADMM method.

When there is a nonseparable term coupling variables together in the objective like \eqref{eq:mb-prob}, existing works usually replace the nonseparable term by a relatively easier majorization function during the iterations and perform the upper-bound or majorized ADMM updates. For example, \cite{hong2014block} considers generally multi-block problems with nonseparable Lipschitz-differentiable term. Under certain error bound conditions and a diminishing dual stepsize assumption, it shows subsequence convergence, i.e., any cluster point of the iterate sequence is a primal-dual solution. Along a similar direction, \cite{cui2015convergence} specializes the method of \cite{hong2014block} to two-block problems and establishes global convergence and also $O(1/t)$ rate with any positive dual stepsize. Without changing the nonseparable term, \cite{gao2015first} adds proximal terms into the augmented Lagrangian function during each update and shows $O(1/t)$ sublinear convergence by assuming strong convexity on the objective. Very recently, \cite{chen2015Ext-ADMM} directly applies ADMM to \eqref{eq:mb-prob} with $m=2$, i.e., only two block variables, and establishes iterate sequence convergence to a solution while no rate estimate has been shown.

\subsection{Jacobian and Gauss-Seidel block coordinate update}
On updating one among the $m$ block variables, the Jacobi method uses the values of all the other blocks from the previous iteration while Gauss-Seidel method always takes the most recent values. For optimization problems without constraint or with block separable constraint, Jacobian BCU enjoys the same convergence as gradient descent or proximal gradient method, and Gauss-Seidel BCU is also guaranteed to converge under mild conditions (e.g., \cite{Tseng-01, XY_2013_multiblock, razaviyayn2013unified, hong2016iteration}). For linearly constrained problems, the Jacobi method converges to optimal value with merely weak convexity \cite{he2015full, deng2016parallel, he2016proximal}. However, the Gauss-Seidel update requires additional assumptions for convergence, though it can empirically perform better than the Jacobi method if it happens to converge. Counterexamples are constructed in \cite{chen2016direct, feng2015cvg-LADMM} to show possible divergence of Gauss-Seidel block coordinate update for linearly constrained problems. To guarantee convergence of the Gauss-Seidel update, many existing works assume strong convexity on the objective or part of it. For example, \cite{han2012note} considers linearly constrained convex programs with separable objective function. It shows the convergence of multiblock ADMM by assuming strong convexity on the objective. Sublinear convergence of multiblock proximal ADMM is established for problems with nonseparable objective in \cite{gao2015first}, which assumes strong convexity on the objective and also chooses parameters dependent on the strong convexity constant. On three-block case, \cite{cai2014directstrong} assumes strong convexity on the third block and also full column-rankness of the last two block coefficient matrices, and \cite{chencai2013convergence} assumes strong convexity on the last two blocks. There are also works that do not assume strong convexity but require some other conditions. For instance, \cite{hong2014block} considers linearly constrained convex problems with nonseparable objective. It shows the convergence of a majorized multiblock ADMM with diminishing dual stepsizes and by assuming a local error bound condition. The work \cite{chen2016direct} assumes orthogonality between two block matrices in the linear constraint and proves the convergence of three-block ADMM by reducing it to the classic two-block case. Intuitively, Jacobian block update is a linearized ALM, and thus in some sense, it is equivalent to performing an augmented dual gradient ascent to the multipliers \cite{liu2016non}. On the contrary, Gauss-Seidel update uses inexact dual gradient, and because the blocks are updated cyclicly, the error can accumulate; see \cite{sun2015expected} where random permutation is performed before block update to cancel the error and iterate convergence in expectation is established.

\subsection{Contributions}   We summarize our contributions as follows.
\begin{itemize}
\item We propose a hybrid Jacobian and Gauss-Seidel BCU method for solving \eqref{eq:mb-prob}. Through affinely combining two most recent iterates, the proposed method can update block variables, in Jacobian or Gauss-Seidel manner, or a mixture of them, Jacobian rules within groups of block variables and a Gauss-Seidel way between groups. It can enjoy the theoretical convergence guarantee as BCU with fully Jacobian update rule and also practical fast convergence as the method with fully Gauss-Seidel rule.
\item We establish global iterate sequence convergence and $O(1/t)$ rate results of the proposed BCU method, by assuming certain conditions on the affinely mixing matrix (see Assumption \ref{assump2}) and choosing appropriate weight matrix $P_i$'s in \eqref{eq:update-x}. Different from the after-correction step proposed in \cite{he2012ADM-GBS} to obtain convergence of multi-block ADMM with fully Gauss-Seidel rule, the affine combination of two iterates we propose can be regarded as a correction step before updating every block variable. In addition, our method allows nonseparable quadratic terms in the objective while the method in \cite{he2012ADM-GBS} can only deal with separable terms. Furthermore, utilizing prox-linearization technique, our method can have much simpler updates.
\item We discuss how to choose the affinely mixing matrix, which can be  determined through solving a semidefinite programming (SDP); see \eqref{eq:u-sdp}. One can choose a desired matrix by adding certain constraints in the SDP. Compared to the original problem, the SDP is much smaller and can be solved offline. We demonstrate that the algorithm with the combining matrix found in this way can perform significantly better than that with all-one matrix (i.e., with fully Jacobian rule).
\item We apply the proposed BCU method to quadratic programming, the compressive principal component pursuit (see \eqref{eq:cpcp} below), and the multi-class support vector machine (see \eqref{eq:msvm} below). By adapting $P_i$'s in \eqref{eq:update-x}, we demonstrate that our method can outperform a randomized BCU method recently proposed in \cite{GXZ-RPDCU2016} and is comparable to the direct ADMM that has no guaranteed convergence. Therefore, the proposed method can be a more reliable and also efficient algorithm for solving problems in the form of \eqref{eq:mb-prob}.
\end{itemize} 

\subsection{Outline}

The rest of the paper is organized as follows. In section \ref{sec:algorithm}, we present our algorithm, and we show its convergence with sublinear rate estimate in section \ref{sec:analysis}. In section \ref{sec:choose-w}, we discuss how to choose the affinely mixing matrices used in the algorithm. Numerical experiments are performed in section \ref{sec:experiment}. Finally, section \ref{sec:conclusion} concludes the paper.

\section{Algorithm}\label{sec:algorithm}
In this section, we present a BCU method for solving \eqref{eq:mb-prob}. Algorithm \ref{alg:hyb-admm} summarizes the proposed method. In the algorithm, $A_i$ denotes the $i$-th block matrix of $A$ corresponding to $x_i$, and 
\begin{equation}
\label{eq:hat-tilde-x}
f(x)=\frac{1}{2}x^\top Qx, \quad \hat{x}_j^{k,i}=x_j^{k+1}-{w}_{ij}(x_j^{k+1}-x_j^k),\ \forall i,j,
\end{equation}
where $w_{ij}$ is the $(i,j)$-th entry of $W$.
\begin{algorithm}\caption{Hybrid Jacobian and Gauss-Seidel proximal block coordinate update for \eqref{eq:mb-prob}}\label{alg:hyb-admm}
\DontPrintSemicolon
\textbf{Initialization:} choose $x^1$ and set $\lambda^1=0$; choose $\beta,\rho,W$ and $P_i\succeq0,\,i=1,\ldots,m$.\;
\For{$k=1,2,\ldots$}{
\For{$i=1,2,\ldots,m$}{
Update $x_i$ by
\begin{align}\label{eq:update-x}
\hspace{-0.5cm}x^{k+1}_i\in & \argmin_{x_i}\left\langle\nabla_i f(\hat{x}^{k,i})-A_i^\top\big(\lambda^k-\beta(A\hat{x}^{k,i}-b)\big), x_i\right\rangle+g_i(x_i)+\frac{1}{2}\|x_i-x_i^k\|_{P_i}^2, 
\end{align}
}
Update the multipliers by 
\begin{equation}\label{eq:update-lam}
\lambda^{k+1}=\lambda^k-\rho(Ax^{k+1}-b),
\end{equation}
\If{a certain stopping condition is satisfied}{
Return $(x^{k+1},\lambda^{k+1})$
}
}
\end{algorithm}

The algorithm is derived by applying a cyclic proximal block coordinate update to the augmented Lagrangian function of \eqref{eq:mb-prob}, that is
\begin{equation}\label{eq:aug-lag}
\cL_\beta(x,\lambda)= f(x)+\sum_{i=1}^m g_i(x_i)-\langle\lambda, Ax-b\rangle+\frac{\beta}{2}\|Ax-b\|^2.
\end{equation} 
At each iteration, we first renew every $x_i$ by minimizing a proximal approximation of $\cL_\beta$,  
one primal block variable at a time while all the remaining ones are fixed, and we then update the multiplier $\lambda$ by a dual gradient ascent step. 

Note that in \eqref{eq:update-x}, for simplicity, we evaluate the partial gradients of the quadratic term $f$ and augmented term $\frac{\beta}{2}\|Ax-b\|^2$ at the same point $\hat{x}^{k,i}$. In general, two different points can be used to explore structures of $Q$ and $A$, and our convergence analysis still holds by choosing two different $W$'s appropriately. 
Practically, one can set 
$$W=\left[\begin{array}{cccc}1 & 1 & \cdots & 1\\ 1 & 1 & \cdots & 1\\ \vdots & \ddots &\ddots & \vdots\\ 1 & \cdots & 1 & 1\end{array}\right],\text{ or } W=\left[\begin{array}{cccc}1 & 1 & \cdots & 1\\ 0 & 1 & \cdots & 1\\ \vdots & \ddots &\ddots & \vdots\\ 0 & \cdots & 0 & 1\end{array}\right],$$
which give fully Jacobian and fully Gauss-Seidel updates respectively. Empirically, the latter one usually performs better. However, theoretically, the latter case may fail to converge when $m\ge 3$, as shown in \cite{chen2016direct, feng2015cvg-LADMM}. Therefore, we will design a mixing matrix $W$ between the above two choices to theoretically guarantee  convergence and also maintain practically nice performance. Our choice is inspired from the convergence analysis and will be discussed in details in section \ref{sec:choose-w}. 

We finish this section by giving examples that fully Gauss-Seidel update method may diverge even if $P_i$'s are large, which strongly motivate our hybrid strategy. Let us consider the problem
$$\min_x~ 0,\st A_1 x_1 +A_2 x_2 +A_3 x_3 = 0,$$
where $A_1=[1, 1-\vareps, 1-\vareps]^\top$, $A_2=[1, 1, 1-\vareps]^\top$ and $A_3=[1,1,1]^\top$. For any $\vareps\neq 0$, $A$ is invertiable and thus the problem has a unique solution $[0,0,0]^\top$. Applied to the above problem with fully Gauss-Seidel update, $\beta=\rho=1$, and $P_i=\frac{\max_j\|A_j\|^2}{\tau},\forall i$, Algorithm \ref{alg:hyb-admm} becomes the following iterative method (see \cite[section 3]{feng2015cvg-LADMM}):  
{\small$$\left[\begin{array}{r}x_1^{k+1}\\x_2^{k+1}\\x_3^{k+1}\\\lambda^{k+1}\end{array}\right]=
\left[\begin{array}{cccc}1&0&0&0_{1\times 3}\\\tau A_2^\top A_1 & 1& 0& 0_{1\times 3}\\\tau A_3^\top A_1 & \tau A_3^\top A_2 & 1 & 0_{1\times 3}\\ A_1 & A_2 &A_3 & I_{3\times 3}\end{array}\right]^{-1}
\left[\begin{array}{cccc}1-\tau A_1^\top A_1 & -\tau A_1^\top A_2 &-\tau A_1^\top A_3 & \tau A_1^\top\\0&1-\tau A_2^\top A_2 &-\tau A_2^\top A_3 & \tau A_2^\top\\0 &0 & 1-\tau A_3^\top A_3 & \tau A_3^\top\\ 0_{3\times 1} & 0_{3\times 1} & 0_{3\times 1} & I_{3\times 3}\end{array}\right]
\left[\begin{array}{r}x_1^{k}\\x_2^{k}\\x_3^{k}\\\lambda^{k}\end{array}\right].$$
}Denote $M_\tau$ as the iterating matrix. Then the algorithm converges if the spectral radius of $M_\tau$ is smaller than one and diverges if larger than one. For $\vareps$ varying among $\{10^{-1}, 10^{-2}, 10^{-3}, 10^{-4}\}$, we search for the largest $\tau$ with initial value $\frac{1}{3}$ and decrement $10^{-5}$ such that the spectral radius of $M_\tau$ is less than one. The results are listed in Table \ref{table:eps-tau} below. They indicate that to guarantee the convergence of the algorithm, a diminishing stepsize would be required for the $x$-update, while note that $\tau$ can be as large as $\frac{1}{3}$ for convergence if the fully Jacobian update is employed.
\begin{table}[H]
\begin{center}
\begin{tabular}{c|cccc}
\hline
$\vareps$ & $10^{-1}$ & $10^{-2}$ & $10^{-3}$ & $10^{-4}$\\
$\tau$ & $1.45473\times 10^{-1}$ & $1.34433\times 10^{-2}$ &  $1.33333\times 10^{-3}$ & $1.33333 \times 10^{-4}$\\\hline
\end{tabular}
\end{center}
\caption{Values of $\vareps$ and the corresponding largest $\tau$ such that spectral radius of $M_\tau$ less than one.}\label{table:eps-tau}
\end{table}

\section{Convergence analysis}\label{sec:analysis}
In this section, we analyze the convergence of Algorithm \ref{alg:hyb-admm}. We establish its global iterate sequence convergence and $O(1/t)$ rate by choosing an appropriate mixing matrix $W$ and assuming merely weak convexity on the problem. 

\subsection{Notation and preliminary results} 

Before proceeding with our analysis, we introduce some notation and a few preliminary lemmas.

We let $$g(x)=\sum_{i=1}^m g_i(x_i),\quad F=f+g.$$ A point $x^*$ is a solution to \eqref{eq:mb-prob} if there exists $\lambda^*$ such that the KKT conditions hold:
\begin{subequations}\label{eq:kkts}
\begin{align}
&0\in\partial F(x^*)-A^\top\lambda^*,\label{eq:kkts-d}\\
&Ax^*-b=0,\label{eq:kkts-p}
\end{align}
\end{subequations}
where $\partial F$ denotes the subdifferential of $F$.
Together with the convexity of $F$, \eqref{eq:kkts} implies
\begin{equation}\label{eq:opt-cond}
F(x)-F(x^*)-\langle\lambda^*, Ax-b\rangle\ge 0,\,\forall x.
\end{equation}
We denote $\cX^*$ as the solution set of \eqref{eq:mb-prob}. For any vector $v$ and any symmetric matrix $S$ of appropriate size, we define $\|v\|_S^2=v^\top Sv$. Note this definition does not require $S$ to be PSD, so $\|v\|_S^2$ may be negative. $I$ is reserved for the identity matrix and $E$ for the all-one matrix, whose size would be clear from the context. $A\otimes B$ represents the Kronecker product of two matrices $A$ and $B$. For any matrices $A,B,C$ and $D$ of appropriate sizes, it holds that (c.f., \cite[Chapter 4]{hom1991topics})
\begin{align}
&(A\otimes B)^\top = A^\top\otimes B^\top,\label{eq:kron-trans}\\
&(A\otimes B)(C\otimes D) = (AC)\otimes (BD).\label{eq:kron-mult}
\end{align}

The following lemmas can be found in \cite{xu2017accelerated-alm} and also appear in \cite{gao2014information, GXZ-RPDCU2016}.
\begin{lemma}
For any two vectors $\vu, \vv$ and a symmetric matrix $\vS$, we have
\begin{equation}\label{uv-cross}
2\vu^\top \vS \vv = \|\vu\|_\vS^2+\|\vv\|_\vS^2-\|\vu-\vv\|_\vS^2.
\end{equation}
\end{lemma}

\begin{lemma}\label{lem:x-rate}
Given (possibly nonconvex) functions $F$, $\phi$, and a fixed point $\bar{x}$, if for any $\lambda$, it holds that
$$F(\bar{x})-F(x^*)-\langle\lambda, A\bar{x}-b\rangle\le \phi(\lambda),$$
then for any $\gamma>0$, we have
$$F(\bar{x})-F(x^*)+\gamma \|A\bar{x}-b\|\le \sup_{\|\lambda\|\le \gamma}\phi(\lambda).
$$
\end{lemma}

\begin{lemma} \label{lem:equiv-rate}
Let $(x^*,\lambda^*)$ be any point satisfying the condition in \eqref{eq:opt-cond}.
If
$F(\bar{x})-F(x^*)+\gamma\|A\bar{x}-b\|\le \epsilon$ for certain $\epsilon\ge0$ and $\gamma >\|\lambda^*\|$,
then
$$\|A\bar{x}-b\| \leq \frac{\epsilon}{\gamma-\|\lambda^*\|} \mbox{ and }  -\frac{\|\lambda^*\|\epsilon}{\gamma-\|\lambda^*\|}\le F(\bar{x})-F(x^*) \leq \epsilon.$$
\end{lemma}

\subsection{Technical assumptions}

Throughout our analysis, we make the following assumptions.

\begin{assumption}\label{assump1}
In \eqref{eq:mb-prob}, $Q$ is PSD, and $g_i$'s are proper closed convex functions. There exists a pair $(x^*,\lambda^*)$ satisfying the KKT conditions in \eqref{eq:kkts}.
\end{assumption}

\begin{assumption}\label{assump2}
The solution set of the subproblem in \eqref{eq:update-x} is nonempty for any $i$.
\end{assumption}

\begin{assumption}\label{assump3}
The mixing matrix $W$ satisfy:
\begin{align}\label{eq:hat-tilde-w}
&w_{ij}=1,\,\forall j\ge i,\\
&\exists\,u\text{ such that }U=W-eu^\top\text{ is symmetric},\label{eq:sym-w}
\end{align}
where $e$ is the all-one vector.
\end{assumption}

The first assumption is minimal for finding a solution to \eqref{eq:mb-prob}. The second one is for well-definedness of the proposed algorithm, and it can be guaranteed if $P_i$'s are all positive definite. 

The requirements in \eqref{eq:hat-tilde-w} are for easy implementation of the update in \eqref{eq:update-x} because otherwise $x^{k+1}_i$ may implicitly depend on the later updated block variables. The conditions in \eqref{eq:sym-w} are for technical reason; see \eqref{eq:relax-ineq} below. They can be satisfied by first choosing $u$ and then determining the corresponding $W$; see the formula \eqref{eq:formula-W} below. How to choose $u$ will be discussed in the next section since it is inspired from our convergence analysis.

\subsection{Convergence results of Algorithm \ref{alg:hyb-admm}}
We show that with appropriate proximal terms, Algorithm \ref{alg:hyb-admm} can have $O(1/t)$ convergence rate, where $t$ is the number of iterations. The result includes several existing ones as special cases, and we will discuss it after presenting our convergence result.

We first establish a few inequalities. Since $Q$ is PSD, there exists a matrix $H$ such that $Q=H^\top H$. Corresponding to the partition of $x$, we let $H=(H_1,\ldots,H_m)$. 

\begin{proposition}\label{prop:cross-term}
Let ${W}$ satisfy the conditions in Assumption \ref{assump3} and define
$$y_i=H_ix_i,\quad z_i=A_ix_i,\,\forall i.$$ Then for any ${\alpha}>0$,
\begin{align}\label{eq:bd1-y}
&\sum_{i=1}^m\sum_{j=1}^m{w}_{ij}\big\langle H_i(x_i^{k+1}-x_i), H_j(x_j^{k+1}-x_j^k)\big\rangle\\
\le&\frac{1}{2}\left(\|y^{k+1}-y\|_{{V}}^2-\|y^k-y\|_{{V}}^2+\|y^{k+1}-y^k\|_{{V}}^2\right)+\frac{1}{2{\alpha}}\|H(x^{k+1}-x)\|^2+\frac{{\alpha}}{2}\left\|\big({u}^\top\otimes I\big)(y^{k+1}-y^k)\right\|^2,\nonumber
\end{align}
and
\begin{align}\label{eq:bd1-z}
&\sum_{i=1}^m\sum_{j=1}^m{w}_{ij}\big\langle A_i(x_i^{k+1}-x_i), A_j(x_j^{k+1}-x_j^k)\big\rangle\\
\le &\frac{1}{2}\left(\|z^{k+1}-z\|_{{V}}^2-\|z^k-z\|_{{V}}^2+\|z^{k+1}-z^k\|_{{V}}^2\right)+\frac{1}{2{\alpha}}\|A(x^{k+1}-b)\|^2+\frac{{\alpha}}{2}\left\|\big({u}^\top\otimes I\big)(z^{k+1}-z^k)\right\|^2,\nonumber
\end{align}
where ${V}={U}\otimes I$.
\end{proposition}
\begin{proof}
We only show \eqref{eq:bd1-y}, and \eqref{eq:bd1-z} follows in the same way. By the definition of $y$, we have
\begin{align}
&\sum_{i=1}^m\sum_{j=1}^m{w}_{ij}\big\langle H_i(x_i^{k+1}-x_i), H_j(x_j^{k+1}-x_j^k)\big\rangle\cr
=&(y^{k+1}-y)^\top\big({W}\otimes I\big)(y^{k+1}-y^k)\cr
\overset{\eqref{eq:sym-w}}=&(y^{k+1}-y)^\top\big({U}\otimes I\big)(y^{k+1}-y^k)+(y^{k+1}-y)^\top\left((e{u}^\top)\otimes I\right)(y^{k+1}-y^k)\cr
\overset{\eqref{eq:kron-mult}}=&(y^{k+1}-y)^\top\big({U}\otimes I\big)(y^{k+1}-y^k)+(y^{k+1}-y)^\top\left((e\otimes I)\big({u}^\top\otimes I\big)\right)(y^{k+1}-y^k)\cr
=&(y^{k+1}-y)^\top\big({U}\otimes I\big)(y^{k+1}-y^k)+\big(H(x^{k+1}-x)\big)^\top\left({u}^\top\otimes I\right)(y^{k+1}-y^k)\label{eq:relax-ineq}\\
\le &(y^{k+1}-y)^\top\big({U}\otimes I\big)(y^{k+1}-y^k)+\frac{1}{2{\alpha}}\|H(x^{k+1}-x)\|^2+\frac{{\alpha}}{2}\left\|\big({u}^\top\otimes I\big)(y^{k+1}-y^k)\right\|^2\cr
\overset{\eqref{uv-cross}}=&\frac{1}{2}\left(\|y^{k+1}-y\|_{{V}}^2-\|y^k-y\|_{{V}}^2+\|y^{k+1}-y^k\|_{{V}}^2\right)+\frac{1}{2{\alpha}}\|H(x^{k+1}-x)\|^2+\frac{{\alpha}}{2}\left\|\big({u}^\top\otimes I\big)(y^{k+1}-y^k)\right\|^2,\nonumber
\end{align}
where the inequality follows from the Cauchy-Schwartz inequality.
\end{proof}

\begin{proposition}[One-iteration result]\label{prop:one-iter}
Let $\{(x^k,\lambda^k)\}_{k\ge 1}$ be the sequence generated from Algorithm \ref{alg:hyb-admm}. Under Assumptions \ref{assump2} and \ref{assump3}, for any $(x,\lambda)$ such that $Ax=b$, it holds that
\begin{align}\label{eq:1iter}
&F(x^{k+1})-F(x)-\langle\lambda, Ax^{k+1}-b\rangle+\frac{1}{2\rho}\left(\|\lambda^{k+1}-\lambda\|^2-\|\lambda^k-\lambda\|^2\right)+\frac{{\alpha}-1}{2{\alpha}}\|H(x^{k+1}-x)\|^2\cr
&+\left(\frac{\beta-\rho}{2}+\frac{\beta({\alpha}-1)}{2{\alpha}}\right)\|Ax^{k+1}-b\|^2+\frac{1}{2}\big( \|x^{k+1}-x\|_P^2-\|x^k-x\|_P^2+\|x^{k+1}-x^k\|_P^2 \big)\cr
\le&\frac{1}{2}\left(\|y^{k+1}-y\|_{{V}}^2-\|y^k-y\|_{{V}}^2+\|y^{k+1}-y^k\|_{{V}}^2\right)+\frac{{\alpha}}{2}\left\|\big({u}^\top\otimes I\big)(y^{k+1}-y^k)\right\|^2\cr
&+\frac{\beta}{2}\left(\|z^{k+1}-z\|_{{V}}^2-\|z^k-z\|_{{V}}^2+\|z^{k+1}-z^k\|_{{V}}^2\right)+\frac{\beta{\alpha}}{2}\left\|\big({u}^\top\otimes I\big)(z^{k+1}-z^k)\right\|^2,
\end{align}
where $$P=\mathrm{blkdiag}(P_1,\ldots,P_m).$$
\end{proposition}
\begin{proof}
From the update \eqref{eq:update-x}, we have the optimality conditions: for $i=1,\ldots,m$,
\begin{align}\label{eq:1step-opt}
0=& \nabla_i f(\hat{x}^{k,i})-A_i^\top\lambda^k+\tilde{\nabla}g_i(x^{k+1}_i)+\beta A_i^\top\big(A\hat{x}^{k,i}-b\big)+P_i(x_i^{k+1}-x_i^k)\cr
=&H_i^\top\left(H x^{k+1}-\sum_{j=1}^m{w}_{ij}H_j(x_j^{k+1}-x_j^k)\right)-A_i^\top\lambda^k+\tilde{\nabla}g_i(x^{k+1}_i)\cr
&+\beta A_i^\top\left((Ax^{k+1}-b)-\sum_{j=1}^m{w}_{ij}A_j(x_j^{k+1}-x_j^k)\right)+P_i(x_i^{k+1}-x_i^k),
\end{align}
where $\tilde{\nabla}g_i(x^{k+1}_i)$ is a subgradient of $g_i$ at $x^{k+1}_i$.
Doing inner product of both sides of \eqref{eq:1step-opt} with $x_i^{k+1}-x_i$, and summing them together over $i$, we have
\begin{align}\label{eq:bd1}
&\sum_{i=1}^m\sum_{j=1}^m{w}_{ij}\big\langle H_i(x_i^{k+1}-x_i), H_j(x_j^{k+1}-x_j^k)\big\rangle+\beta\sum_{i=1}^m\sum_{j=1}^m{w}_{ij}\big\langle A_i(x_i^{k+1}-x_i), A_j(x_j^{k+1}-x_j^k)\big\rangle\cr
=&\big\langle H(x^{k+1}-x), Hx^{k+1}\big\rangle+\big\langle A(x^{k+1}-x), -\lambda^k+\beta(Ax^{k+1}-b)\big\rangle+\big\langle x^{k+1}-x,\tilde{\nabla} g(x^{k+1})\big\rangle\cr
&+\big\langle x^{k+1}-x, P(x^{k+1}-x^k)\big\rangle\cr
\ge &\big\langle H(x^{k+1}-x), Hx^{k+1}\big\rangle+\big\langle A(x^{k+1}-x), -\lambda^k+\beta(Ax^{k+1}-b)\big\rangle+g(x^{k+1})-g(x)\cr
&+\big\langle x^{k+1}-x, P(x^{k+1}-x^k)\big\rangle\cr
= & \frac{1}{2}\big(\|H(x^{k+1}-x)\|^2-\|Hx\|^2+\|Hx^{k+1}\|^2\big)-\langle\lambda^{k+1}, Ax^{k+1}-b\rangle\cr
&+(\beta-\rho)\|Ax^{k+1}-b\|^2+g(x^{k+1})-g(x)+\frac{1}{2}\big( \|x^{k+1}-x\|_P^2-\|x^k-x\|_P^2+\|x^{k+1}-x^k\|_P^2 \big)\cr
=&\frac{1}{2}\|H(x^{k+1}-x)\|^2 +F(x^{k+1})-F(x)-\langle\lambda^{k+1}, Ax^{k+1}-b\rangle\cr
&+(\beta-\rho)\|Ax^{k+1}-b\|^2+\frac{1}{2}\big( \|x^{k+1}-x\|_P^2-\|x^k-x\|_P^2+\|x^{k+1}-x^k\|_P^2 \big),
\end{align}
where the inequality uses the convexity of $g$, and in the second equality, we have used \eqref{uv-cross}, the update rule \eqref{eq:update-lam}, and the condition $Ax=b$.

Substituting \eqref{eq:bd1-y} and \eqref{eq:bd1-z} into \eqref{eq:bd1}, we have
\begin{align}\label{eq:bd2}
&F(x^{k+1})-F(x)-\langle\lambda^{k+1}, Ax^{k+1}-b\rangle+\frac{{\alpha}-1}{2{\alpha}}\|H(x^{k+1}-x)\|^2\cr
&+\left(\frac{\beta}{2}+\frac{\beta({\alpha}-1)}{2{\alpha}}-\rho\right)\|Ax^{k+1}-b\|^2+\frac{1}{2}\big( \|x^{k+1}-x\|_P^2-\|x^k-x\|_P^2+\|x^{k+1}-x^k\|_P^2 \big)\cr
\le&\frac{1}{2}\left(\|y^{k+1}-y\|_{{V}}^2-\|y^k-y\|_{{V}}^2+\|y^{k+1}-y^k\|_{{V}}^2\right)+\frac{{\alpha}}{2}\left\|\big({u}^\top\otimes I\big)(y^{k+1}-y^k)\right\|^2\cr
&+\frac{\beta}{2}\left(\|z^{k+1}-z\|_{{V}}^2-\|z^k-z\|_{{V}}^2+\|z^{k+1}-z^k\|_{{V}}^2\right)+\frac{\beta{\alpha}}{2}\left\|\big({u}^\top\otimes I\big)(z^{k+1}-z^k)\right\|^2.
\end{align}
From the update \eqref{eq:update-lam}, we have
\begin{align}\label{eq:eq-lam}
0=&\langle\lambda^{k+1}-\lambda, Ax^{k+1}-b\rangle+\frac{1}{\rho}\langle\lambda^{k+1}-\lambda,\lambda^{k+1}-\lambda^k\rangle\cr
\overset{\eqref{uv-cross}}=&\langle\lambda^{k+1}-\lambda, Ax^{k+1}-b\rangle+\frac{1}{2\rho}\left(\|\lambda^{k+1}-\lambda\|^2-\|\lambda^k-\lambda\|^2+\|\lambda^{k+1}-\lambda^k\|^2\right)\cr
=&\langle\lambda^{k+1}-\lambda, Ax^{k+1}-b\rangle+\frac{1}{2\rho}\left(\|\lambda^{k+1}-\lambda\|^2-\|\lambda^k-\lambda\|^2\right)+\frac{\rho}{2}\|Ax^{k+1}-b\|^2.
\end{align}
Summing \eqref{eq:bd2} and \eqref{eq:eq-lam} together gives the desired result.
\end{proof}

Now we are ready to present our main result.
\begin{theorem}\label{thm:sub-rate}
Under Assumptions \ref{assump1} through \ref{assump3}, let $\{(x^k,\lambda^k)\}_{k\ge 1}$ be the sequence generated from Algorithm \ref{alg:hyb-admm} with parameters:
\begin{align}
&\beta\ge \rho>0,\label{eq:beta-rho}\\
&P-D_H^\top \big(({W}-e{u}^\top+{\alpha}{u}{u}^\top)\otimes I\big) D_H -\beta D_A^\top \big(({W}-e{u}^\top+{\alpha}{u}{u}^\top)\otimes I\big) D_A:=\hat{P}\succeq0.\label{eq:cond-P}
\end{align}
where ${\alpha}\ge 1$, and
$$D_H = \mathrm{blkdiag}(H_1,\ldots,H_m),\quad D_A = \mathrm{blkdiag}(A_1,\ldots,A_m).$$
Let $\bar{x}^{t+1}=\sum_{k=1}^t\frac{x^{k+1}}{t}$. Then
\begin{subequations}\label{eq:rate-obj-feas}
\begin{align}
&\big|F(\bar{x}^{t+1})-F(x^*)\big|\le\frac{1}{2t}\left(\frac{\max\{(1+\|\lambda^*\|)^2, 4\|\lambda^*\|^2\}}{\rho}+\|x^1-x^*\|_P^2-\|y^1-y^*\|_{{V}}^2-\beta\|z^1-z^*\|_{{V}}^2\right),\\
&\|A\bar{x}^{t+1}-b\|\le\frac{1}{2t}\left(\frac{\max\{(1+\|\lambda^*\|)^2, 4\|\lambda^*\|^2\}}{\rho}+\|x^1-x^*\|_P^2-\|y^1-y^*\|_{{V}}^2-\beta\|z^1-z^*\|_{{V}}^2\right),
\end{align}
\end{subequations}
where ${V}$ is defined in Proposition \ref{prop:cross-term}, and $(x^*,\lambda^*)$ is any point satisfying the KKT conditions in \eqref{eq:kkts}.

In addition, if ${\alpha}> 1$ and $\hat{P}\succ 0$, then $(x^k,\lambda^k)$ converges to a point $(x^\infty,\lambda^\infty)$ that satisfies the KKT conditions in \eqref{eq:kkts}.
\end{theorem}

\begin{proof}
Summing the inequality \eqref{eq:1iter} from $k=1$ through $t$ and noting $\beta\ge\rho$, we have
\begin{align}\label{eq:bd3}
&\sum_{k=1}^t\big[F(x^{k+1})-F(x)-\langle\lambda, Ax^{k+1}-b\rangle\big]+\frac{1}{2\rho}\|\lambda^{t+1}-\lambda\|^2+\frac{1}{2}\|x^{t+1}-x\|_P^2\cr
&+\sum_{k=1}^t\left(\frac{1}{2}\|x^{k+1}-x^k\|_P^2 +\frac{{\alpha}-1}{2{\alpha}}\|H(x^{k+1}-x)\|^2+\frac{\beta({\alpha}-1)}{2{\alpha}}\|Ax^{k+1}-b\|^2\right)\cr
\le&\frac{1}{2\rho}\|\lambda^1-\lambda\|^2+\frac{1}{2}\|x^1-x\|_P^2-\frac{1}{2}\|y^1-y\|_{{V}}^2-\frac{\beta}{2}\|z^1-z\|_{{V}}^2\\
&+\frac{1}{2}\|y^{t+1}-y\|_{{V}}^2+\frac{1}{2}\sum_{k=1}^t\|y^{k+1}-y^k\|_{{V}}^2+\frac{{\alpha}}{2}\sum_{k=1}^t\left\|\big({u}^\top\otimes I\big)(y^{k+1}-y^k)\right\|^2\cr
&+\frac{\beta}{2}\|z^{t+1}-z\|_{{V}}^2+\frac{\beta}{2}\sum_{k=1}^t\|z^{k+1}-z^k\|_{{V}}^2+\frac{\beta{\alpha}}{2}\sum_{k=1}^t\left\|\big({u}^\top\otimes I\big)(z^{k+1}-z^k)\right\|^2.\nonumber
\end{align}

Note that
\begin{align*}
&\|y^{k+1}-y^k\|_{{V}}^2+{\alpha}\left\|\big({u}^\top\otimes I\big)(y^{k+1}-y^k)\right\|^2\cr
=&(y^{k+1}-y^k)^\top ({V}+{\alpha}{u}{u}^\top\otimes I)(y^{k+1}-y^k)\cr
=&(x^{k+1}-x^k)^\top D_H^\top \big(({W}-e{u}^\top+{\alpha}{u}{u}^\top)\otimes I\big) D_H(x^{k+1}-x^k),
\end{align*}
and similarly
\begin{align*}
&\|z^{k+1}-z^k\|_{{V}}^2+{\alpha}\left\|\big({u}^\top\otimes I\big)(z^{k+1}-z^k)\right\|^2\cr
=&(x^{k+1}-x^k)^\top D_A^\top \big(({W}-e{u}^\top+{\alpha}{u}{u}^\top)\otimes I\big) D_A(x^{k+1}-x^k).
\end{align*}
Hence, by the choice of $P$ in \eqref{eq:cond-P},  we have from \eqref{eq:bd3} that
\begin{align}\label{eq:bd3-sum}
&\sum_{k=1}^t\big[F(x^{k+1})-F(x)-\langle\lambda, Ax^{k+1}-b\rangle\big]+\frac{1}{2\rho}\|\lambda^{t+1}-\lambda\|^2\cr
&+\sum_{k=1}^t\left(\frac{1}{2}\|x^{k+1}-x^k\|_{\hat{P}}^2+\frac{{\alpha}-1}{2{\alpha}}\|H(x^{k+1}-x)\|^2+\frac{\beta({\alpha}-1)}{2{\alpha}}\|Ax^{k+1}-b\|^2\right)\cr
\le&\frac{1}{2\rho}\|\lambda^1-\lambda\|^2+\frac{1}{2}\|x^1-x\|_P^2-\frac{1}{2}\|y^1-y\|_{{V}}^2-\frac{\beta}{2}\|z^1-z\|_{{V}}^2.
\end{align}

Since ${\alpha}\ge 1$ and $\hat{P}\succeq 0$, it follows from the above inequality and the convexity of $F$ that
\begin{align}\label{eq:bd4}
&F(\bar{x}^{t+1})-F(x^*)-\langle\lambda, A\bar{x}^{t+1}-b\rangle\\
\le&\frac{1}{2t}\left(\frac{1}{\rho}\|\lambda^1-\lambda\|^2+\|x^1-x^*\|_P^2-\|y^1-y^*\|_{{V}}^2-\beta\|z^1-z^*\|_{{V}}^2\right).\nonumber
\end{align}
Since $\lambda^1=0$, we use Lemmas \ref{lem:x-rate} and \ref{lem:equiv-rate} with $\gamma=\max\{1+\|\lambda^*\|, 2\|\lambda^*\|\}$ to have \eqref{eq:rate-obj-feas}.

If ${\alpha}>1$ and $\hat{P}\succ0$, then letting $x=x^*, \lambda=\lambda^*$ in  \eqref{eq:bd3-sum} and also using \eqref{eq:opt-cond}, we have 
\begin{equation}\label{eq:diff-lim}
\lim_{k\to\infty}(x^{k+1}-x^k)=0,\quad \lim_{k\to\infty}(\lambda^{k+1}-\lambda^k)=-\rho\lim_{k\to\infty}(Ax^{k+1}-b)=0,
\end{equation}
and thus
\begin{equation}\label{eq:lim-hat-tilde}
\lim_{k\to\infty}(x^k-\hat{x}^{k,i})=0,\,\forall i.
\end{equation}
On the other hand, letting $(x,\lambda)=(x^*,\lambda^*)$ in \eqref{eq:1iter}, using \eqref{eq:opt-cond}, and noting ${\alpha},{\alpha}\ge1$, we have
\begin{align}\label{eq:1iter-bd}
&\frac{1}{2\rho}\|\lambda^{k+1}-\lambda^*\|^2+\frac{1}{2}\left(\|x^{k+1}-x^*\|_P^2-\|y^{k+1}-y^*\|_{{V}}^2-\beta\|z^{k+1}-z^*\|_{{V}}^2\right)\cr
\le&\frac{1}{2\rho}\|\lambda^k-\lambda^*\|^2+\frac{1}{2}\left(\|x^k-x^*\|_P^2-\|y^k-y^*\|_{{V}}^2-\beta\|z^k-z^*\|_{{V}}^2\right),
\end{align}
which together with the choice of $P$ indicates the boundedness of $\{(x^k,\lambda^k)\}_{k\ge1}$. Hence, it must have a finite cluster point $(x^\infty,\lambda^\infty)$, and there is a subsequence $\{(x^k,\lambda^k)\}_{k\in\cK}$ convergent to this cluster point. From \eqref{eq:diff-lim}, it immediately follows that $Ax^\infty-b=0$. In addition, letting $\cK\ni k\to\infty$ in \eqref{eq:update-x} and using \eqref{eq:diff-lim} and \eqref{eq:lim-hat-tilde} gives
$$x_i^\infty =\argmin_{x_i}\langle\nabla_i f(x^\infty)-A_i^\top\lambda^\infty,x_i\rangle+g_i(x_i)+\frac{1}{2}\|x_i-x_i^\infty\|_{P_i}^2,\,\forall i,$$
and thus the optimality condition holds:
$$0\in \nabla_i f(x^\infty)+\partial g_i(x_i^\infty)-A_i^\top\lambda^\infty,\,\forall i.$$
Therefore, $(x^\infty,\lambda^\infty)$ satisfies the conditions in \eqref{eq:kkts}. Since \eqref{eq:1iter-bd} holds for any point $(x^*,\lambda^*)$ satisfying \eqref{eq:kkts}, it also holds with $(x^*,\lambda^*)=(x^\infty,\lambda^\infty)$. {Denote $$v=(x,\lambda),\qquad S=\left[\begin{array}{cc}P-D_H^\top{V}D_H-\beta D_A^\top{V} D_A & 0\\[0.1cm] 0 & \frac{I}{\rho}\end{array}\right].$$ Then letting $(x^*,\lambda^*)=(x^\infty,\lambda^\infty)$ in \eqref{eq:1iter-bd}, we have $\|v^{k+1}-v^\infty\|_S\le \|v^{k}-v^\infty\|_S$. From \eqref{eq:cond-P} and $\hat{P}\succ0$, it follows that $S\succ 0$, and hence $v^k$ gets closer to $v^\infty$ as $k$ increases.} Because $(x^\infty,\lambda^\infty)$ is a cluster point of $\{(x^k,\lambda^k)\}_{k\ge1}$, we obtain the convergence of $(x^k,\lambda^k)$ to $(x^\infty,\lambda^\infty)$ and complete the proof.
\end{proof}

\section{How to choose a mixing matrix}\label{sec:choose-w}
In this section, we discuss how to choose ${W}$ such that it satisfies Assumption \ref{assump3}. 
Note that the upper triangular part of $W$ has been fixed, and we only need to set its strictly lower triangular part. Denote $\cU(W)$ and $\cL(W)$ respectively as the upper and strictly lower triangular parts of $W$, i.e., $W=\cL(W)+\cU(W)$, and thus \eqref{eq:sym-w} is equivalent to requiring the existence of $u$ such that
$$\cL(W)+\cU(W)-eu^\top=\cL(W)^\top+\cU(W)^\top-ue^\top.$$ 
It suffices to let
\begin{equation}\label{eq:formula-W}
\cL(W)=\cL(W^\top+eu^\top-u e^\top)= \cL(eu^\top-u e^\top+E).
\end{equation}
Therefore, given any vector $u$, we can find a corresponding $W$ by setting its upper triangular part to all one's and its strictly lower triangular part according to the above formula.

\subsection{Finding $u$ by solving SDP}\label{sec:opt-w}
Theoretically, proximal terms help convergence guarantee of the algorithm. However, empirically, these terms can slow the convergence speed. Based on these observations, we aim at finding a block diagonal $P$ such that \eqref{eq:cond-P} holds and also is as close to \emph{zero} as possible.

One choice of $P$ satisfying \eqref{eq:cond-P} could be
\begin{equation}\label{eq:choice-P}
P_i = (1-{D}_i)H_i^\top H_i + \beta(1-{D}_i)A_i^\top A_i + d(\|H_i\|_2^2+\beta\|A_i\|_2^2) I,\,i=1,\ldots m,
\end{equation}
where ${D}=\mathrm{diag}({D}_1,\ldots,{D}_m)$ with each ${D}_i\in\{0,1\}$ for each $i$, and
\begin{equation}\label{eq:diag-d}
dI\succeq (1+\beta)({W}-e{u}^\top+{\alpha}{u}{u}^\top)-(1+\beta)(I-{D}).
\end{equation}
Note that if $P_i=\eta_i I + H_i^\top H_i$, then \eqref{eq:update-x} reduces to 
\begin{equation}\label{eq:nl-f}
x_i^{k+1}\in\argmin_{x_i} f(x_i, \hat{x}^{k,i}_{\neq i})- \left\langle A_i^\top\big(\lambda^k-\beta(A\hat{x}^{k,i}-b)\big), x_i\right\rangle+g_i(x_i)+\frac{\eta_i}{2}\|x_i-x_i^k\|^2.
\end{equation}
Hence, ${D}_i=0$ indicates no linearization to $f$ or $\|Ax-b\|^2$ at $x_i^k$, and $D_i=1$ indicates linearization to them. If \eqref{eq:nl-f} is easy to solve, one can set $D_i=0$. Otherwise, $D_i=1$ is recommended to have easier subproblems.

With $D$ fixed, to obtain $P$ according to \eqref{eq:choice-P}, we only need to specify the value of $d$. Recall that we aim at finding a close-to-zero $P$, so it would be desirable to choose ${u}$ such that $d$ is as small as possible. For simplicity, we set $\alpha=1$. 
Therefore, to minimize $d$, we solve the following optimization problem:
\begin{equation}\label{eq:find-u}
\min_{u, W} \lambda_{\max}\left(W-eu^\top +uu^\top -I +D\right), \st W-eu^\top = W^\top-ue^\top,
\end{equation} 
where $\lambda_{\max}(B)$ denotes the maximum eigenvalue of a symmetric matrix $B$.

Using \eqref{eq:formula-W}, we represent $W$ by $u$ and write \eqref{eq:find-u} equivalently to
\begin{equation}\label{eq:find-u2}
\min_{u} \lambda_{\max}\left(\cL(eu^\top-u e^\top)+E-eu^\top +uu^\top -I +D\right),
\end{equation} 
which can be further formulated as an SDP by the relation between the positive-definiteness of a $2\times 2$ block matrix and its Schur complement (c.f., \cite[Appendix A.5.5]{boyd2004convex}):
\begin{equation}\label{eq:schur-c}\left[\begin{array}{cc}A & B\\ B^\top & C\end{array}\right]\succeq 0 \Leftrightarrow A-BC^{-1}B^\top \succeq 0,
\end{equation}
where $A$ is symmetric and $C\succ0$. Let
$$\sigma=\lambda_{\max}\left(\cL(eu^\top-u e^\top)+E-eu^\top +uu^\top -I +D\right).$$
Then
$$\sigma I\succeq \cL(eu^\top-u e^\top)+E-eu^\top +uu^\top -I +D,$$
and by \eqref{eq:schur-c} it is equivalent to
$$\left[\begin{array}{cc}(\sigma+1) I-D-\cL(eu^\top-u e^\top)-E+eu^\top & u\\ u^\top & 1\end{array}\right]\succeq 0.
$$
Therefore, \eqref{eq:find-u2} is equivalent to the SDP:
\begin{equation}\label{eq:u-sdp}
\min_{\sigma, u}~ \sigma, \st \left[\begin{array}{cc}(\sigma+1) I-D-\cL(eu^\top-u e^\top)-E+eu^\top & u\\ u^\top & 1\end{array}\right]\succeq 0.
\end{equation}
Note that the problem \eqref{eq:mb-prob} can be extremely large. However, the dimension (i.e., $m$) of the SDP \eqref{eq:u-sdp} could be much smaller (see examples in section \ref{sec:experiment}) and can be efficiently and accurately solved by the interior-point method. In addition, \eqref{eq:u-sdp} does not depend on the data matrix $H$ and $A$, so we can solve it offline. 

If the sizes of $H$ and $A$ are not large, upon solving \eqref{eq:u-sdp}, one can explicitly form the matrix $D_H^\top \big((W-eu^\top+uu^\top)\otimes I\big) D_H + \beta D_A^\top \big((W-eu^\top+uu^\top)\otimes I\big) D_A$ and compute its spectral norm. This way, one can have a smaller $P$. However, for large-scale $H$ or $A$, it can be overwhelmingly expensive to do so.

In addition, note that we can add more constraints to \eqref{eq:u-sdp} to obtain a desired $W$. For instance, we can partition the $m$ blocks into several groups. Then we update the blocks in the same group in parallel in Jacobian manner and cyclically renew the groups. This corresponds to fix a few block matrices in the lower triangular part of $W$ to all one's. 

\subsection{Special cases}
A few special cases are as follows.
\begin{itemize}
\item If $u$ is the \emph{zero} vector, then by \eqref{eq:formula-W}, we have the lower triangular part of $W$ to be all one's, and this way gives a fully Jacobian BCU method. Hence, Theorem \ref{thm:sub-rate} applies for the Jacobian update method. In this case, if ${D}=I$ in \eqref{eq:diag-d}, then we have
$d \ge (1+\beta)m$ that is significantly greater than the optimal value of \eqref{eq:u-sdp}.

\item If we enforce the lower triangular part of $W$ to all \emph{zero}'s, Algorithm \ref{alg:hyb-admm} will reduce to a fully Gauss-Seidel BCU method. However, adding such constraints into \eqref{eq:u-sdp} would lead to infeasibility. Hence, Theorem \ref{thm:sub-rate} does not apply to this case.

\item If $m=2$ and $D=0$, i.e., there are only two blocks and no linearization is performed, solving \eqref{eq:u-sdp} would give the solution $\sigma=0$ and $u=(0,1)^\top$. This way, we have $$W=\left[\begin{array}{cc}1 & 1\\ 0 & 1\end{array}\right],\quad P = \left[\begin{array}{cc}H_1^\top H_1+\beta A_1^\top A_1 & 0\\ 0 &H_2^\top H_2+\beta A_2^\top A_2\end{array}\right]$$
and thus recover the 2-block ADMM with nonseparable quadratic term $f$ in the objective. Theorem \ref{thm:sub-rate} implies $O(1/t)$ convergence rate for this case, and it improves the result in \cite{chen2015Ext-ADMM}, which shows convergence of this special case but without rate estimate.
\end{itemize}

\subsection{Different mixing matrices}
We can choose two different $W$'s to explore the structures of $A$ and $Q$. Let us give an example to illustrate this. Suppose $A$ is a generic matrix and $Q$ block tridiagonal. The $W$ used to linearize the augmented term $\frac{\beta}{2}\|Ax-b\|^2$ can be set in the way as discussed in section \ref{sec:opt-w}. For the mixing matrix to $Q$, we note $H_i^\top H_j=0$ if $|i-j|>1$. Then the left hand side of \eqref{eq:bd1-y} becomes $\sum_{|i-j|\le 1}w_{ij}\big\langle H_i(x_i^{k+1}-x_i), H_j(x_j^{k+1}-x_j^k)\big\rangle$. Hence, following our analysis, we would require there exists $\hat{u}$ such that $\hat{W}-e^\top \hat{u}$ is symmetric, where $\hat{w}_{ij}=0$ if $|i-j|>1$ and $\hat{w}_{ij}=1$ if $i\le j\le i+1$. To completely determine $\hat{W}$, we only need to set its values on the subdiagonal. Similar to \eqref{eq:find-u}, we can find $\hat{u}$ and $\hat{w}_{i+1,i}$'s through solving
$$
\begin{aligned}
&\min_{\hat{u}, \hat{W}} \lambda_{\max}\left(\hat{W}-e\hat{u}^\top +\hat{u}\hat{u}^\top -I +D\right), \\
&\st \hat{W}-e\hat{u}^\top = \hat{W}^\top-\hat{u}e^\top, \hat{w}_{ij}=0, \forall |i-j|>1,\, \hat{w}_{ij}=1, i\le j\le i+1.
\end{aligned}
$$
The optimal value of the above problem is significantly smaller than that of \eqref{eq:find-u}, and thus in \eqref{eq:choice-P} the coefficient before $\|H_i\|^2$ can be set smaller. This way, we will have a smaller $P_i$'s, which can potentially make the algorithm converge faster.

\section{Numerical experiments}\label{sec:experiment}
In this section, we apply Algorithm \ref{alg:hyb-admm} to three problems: quadratic programming, compressive principal component pursuit, and the multi-class support vector machine problem. We test it with two different mixing matrices: all-one matrix and the one given by the method discussed in section \ref{sec:opt-w}. The former corresponds to a fully Jacobian update method and the latter to a hybrid Jacobian and Gauss-Seidel method, dubbed as Jacobi-PC and JaGS-PC respectively. We compare them to a recently proposed randomized proximal block coordinate update method (named as random-PC) in \cite{GXZ-RPDCU2016}, the ADMM with Gauss back substitution (named as ADMM-GBS) in \cite{he2012ADM-GBS}, and also the direct ADMM. Note that the direct ADMM is not guaranteed to converge for problems with more than two blocks, but empirically it can often perform well. ADMM-GBS is designed for separable multi-block convex problems. It does not allow linearization to the augmented term. In addition, it requres all $A_i$'s to be full-column rank. Hence, the proposed algorithm is applicable to broader class of problems, but we observe that JaGS-PC can be comparable to direct ADMM and ADMM-GBS.

We choose to compare with random-PC, direct ADMM, and ADMM-GBS because as well as the proposed methods, all of them have low per-iteration complexity and low memory requirement and belong to inexact ALM framework. On solving the three problems, one can also apply some other methods such as the interior-point method and the projected subgradient method. The interior-point method can converge faster than the proposed ones in terms of iteration number. However, its per-iteration complexity is much higher, and thus total running time can be longer (that is observed for the quadratic programming). In addition, it has a high demand on machine memory and may be inapplicable for large-scale problems such as the compressive principal component pursuit. The projected subgradient method has similar per-iteration complexity as the proposed ones but converges much slower.

In all our tests, we report the results based the actual iterate $x^k$, which is guaranteed to converge to an optimal solution. Although the convergence rate in Theorem \ref{thm:sub-rate} is based on the averaged point $\bar{x}^{t+1}$ (i.e., in ergodic sense), numerically we notice that the convergence speed based on the iterate $x^k$ is often faster than that based on the averaged point. This phenomenon also happens to the classic two-block ADMM. The work \cite{he2012rate-DRS} shows that ADMM has an ergodic sublinear convergence rate but all applications of ADMM still use the actual iterate as the solution.  

\subsection{Adaptive proximal terms}\label{sec:adaptive}
As we mentioned previously, the proximal terms used in \eqref{eq:update-x} help the convergence guarantee but can empirically slow the convergence speed (see Figures \ref{fig:nonadp-qp} and \ref{fig:qp}). Here, we set $P_i$'s similar to \eqref{eq:choice-P} but with a simple adaptive way as follows:
\begin{equation}\label{eq:adap-P}P_i^k = (1-{D}_i)H_i^\top H_i + \beta(1-{D}_i)A_i^\top A_i + d^k(\|H_i\|_2^2+\beta\|A_i\|_2^2) I,\,i=1,\ldots m.
\end{equation}
After each iteration $k$, we check if the following inequality holds
\begin{align}\label{eq:check-ineq}
\eta\|x^{k+1}-x^k\|_{P^k}^2\le & \|y^{k+1}-y^k\|_{{V}}^2+\big\|({u}^\top\otimes I)(y^{k+1}-y^k)\big\|^2\\
&+\beta \|z^{k+1}-z^k\|_{{V}}^2+\beta\big\|({u}^\top\otimes I)(z^{k+1}-z^k)\big\|^2,\nonumber
\end{align}
and set
\begin{equation}\label{eq:adap-d}
d^k=\left\{\begin{array}{ll}\min\left(d^{k-1}+d_{\mathrm{inc}}, d_{\max}\right), &\text{if }\eqref{eq:check-ineq}\text{ holds}\\[0.2cm]
 d^{k-1}, &\text{otherwise} \end{array}\right.
 \end{equation}
 where $d_{\mathrm{inc}}$ is a small positive number, and $\eta=0.999$ is used\footnote{In the proof of Theorem \ref{thm:sub-rate}, we bound the $y$ and $z$-terms by $x$-term. If the left hand side of \eqref{eq:check-ineq} with $\eta<1$ can upper bound the right hand side, then Theorem \ref{thm:sub-rate} guarantees the convergence of $x^k$ to an optimal solution. Numerically, taking $\eta$ close to 1 would make the algorithm more efficient.} in all the tests. For stability and also efficiency, we choose $d^1$ and $d_{\mathrm{inc}}$ such that \eqref{eq:check-ineq} happens not many times. Specifically, we first run the algorithm to 20 iterations with $(d^1, d_{\mathrm{inc}})$ selected from $\{0,0.5,1\}\times\{0.01, 0.1\}$. If there is one pair of values such that \eqref{eq:check-ineq} does not always hold within the 20 iterations\footnote{We notice that if \eqref{eq:check-ineq} happens many times, the iterate may be far away from the optimal solution in the beginning, and that may affect the overall performance; see Figures \ref{fig:pcp-Jac} and \ref{fig:svm-Jac}.}, we accept that pair of $(d^1, d_{\mathrm{inc}})$. Otherwise, we simply set $d^k=d_{\max},\,\forall k$. 
For Jacobi-PC, we set $d_{\max}=\lambda_{\max}(E-I+D)$, and for JaGS-PC, we set $d_{\max}$ to the optimal value of \eqref{eq:u-sdp}, which is solved by SDPT3 \cite{toh1999sdpt3} to high accuracy with stopping tolerance $10^{-12}$. Note that as long as $d_{\mathrm{inc}}$ is positive, $d^k$ can only be incremented in finitely many times, and thus \eqref{eq:check-ineq} can only happen in finitely many iterations. In addition, note that both sides of \eqref{eq:check-ineq} can be evaluated as cheaply as computing $x^\top Q x$ and $x^\top A^\top A x$.

The above adaptive way of setting $P^k$ is inspired from our convergence analysis. If after $k_0$ iterations, \eqref{eq:check-ineq} never holds. Then we can also show a sublinear convergence result by summing \eqref{eq:1iter} from $k=k_0$ through $t$ and then following the same arguments as those in the proof of Theorem \ref{thm:sub-rate}. 

\subsection{Quadratic programming}
We test Jacobi-PC and JaGS-PC on the nonnegative linearly constrained quadratic programming
\begin{equation}\label{eq:qp}
\min_x F(x)=\frac{1}{2} x^\top Q x + c^\top x, \st Ax=b,\, x\ge 0,
\end{equation}
where $Q\in \RR^{n\times n}$ is a symmetric PSD matrix, $A\in\RR^{p\times n}$ and $b\in\RR^p, c\in\RR^n$. In the test, we set $p=200, n=2000$  and $Q = H^\top H$ with $H\in\RR^{(n-10)\times n}$ generated according to the standard Gaussian distribution. This generated $Q$ is degenerate, and thus the problem is only weakly convex. The entries of $c$ follow i.i.d. standard Gaussian distribution and those of $b$  from uniform distribution on $[0,1]$. We set $A=[B,I]$ to guarantee feasibility of the problem with $B$ generated according to standard Gaussian distribution.

We evenly partition the variable $x$ into $m=40$ blocks, each one consisting of 50 coordinates. The same values of parameters are set for both Jacobi-PC and JaGS-PC as follows:
$$\beta=\rho=1,\, {D} = I,\, d^1=0.5,\, d_{\mathrm{inc}}=0.1.$$
They are compared to random-PC that uses the same penalty parameter $\beta=1$ and $\rho=\frac{1}{m}$ according to the analysis in \cite{GXZ-RPDCU2016}. We also adaptively increase the proximal parameter of random-PC in a way similar to that in section \ref{sec:adaptive}. All three methods run to 500 epochs, where each epoch is equivalent to updating all blocks one time. Their per-epoch complexity is almost the same. To evaluate their performance, we compute the distance of the objective to optimal value $|F(x^k)-F(x^*)|$ and the violation of feasibility $\|Ax^k-b\|$ at each iteration $k$, where the optimal solution is obtained by MATLAB solver \verb|quadprog| with ``interior-point-convex'' option. Figures \ref{fig:nonadp-qp} and \ref{fig:qp} plot the results by the three methods in terms of both iteration number and also running time (sec). In Figure \ref{fig:nonadp-qp}, we simply set $d^k=d_{\max}$ for Jacobi-PC and JaGS-PC, i.e., without adapting the proximal terms, where in this test, $d_{\max}=18.3273$ for JaGS-PC and $d_{\max}=40$ for Jacobi-PC. From the figures, we see that JaGS-PC is significantly faster than Jacobi-PC in terms of both objective and feasibility for both adaptive and nonadaptive cases, and random-PC is slightly slower than JaGS-PC.
\begin{figure}
\begin{center}
\begin{tabular}{cc}
\includegraphics[width=0.35\textwidth]{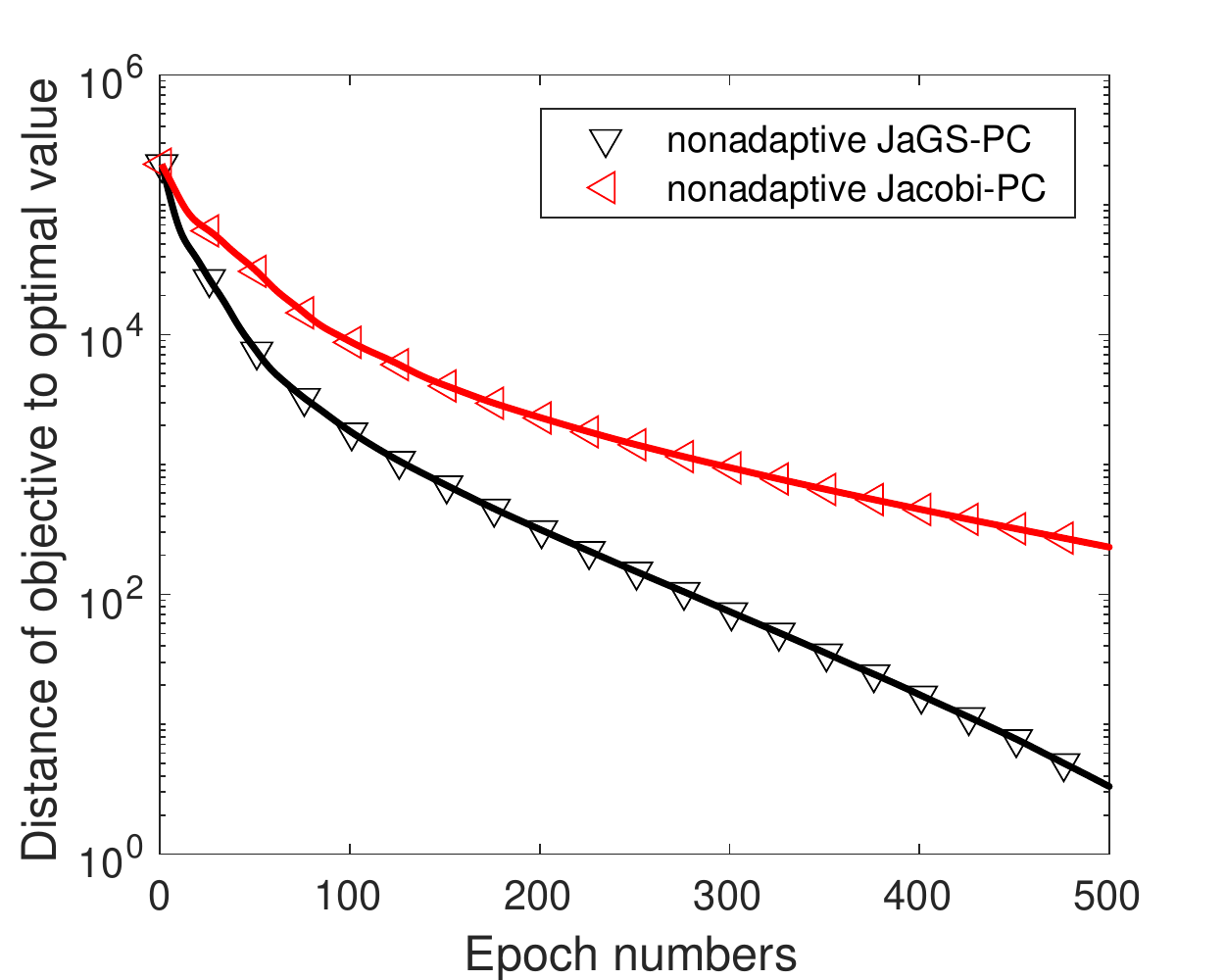} &
\includegraphics[width=0.35\textwidth]{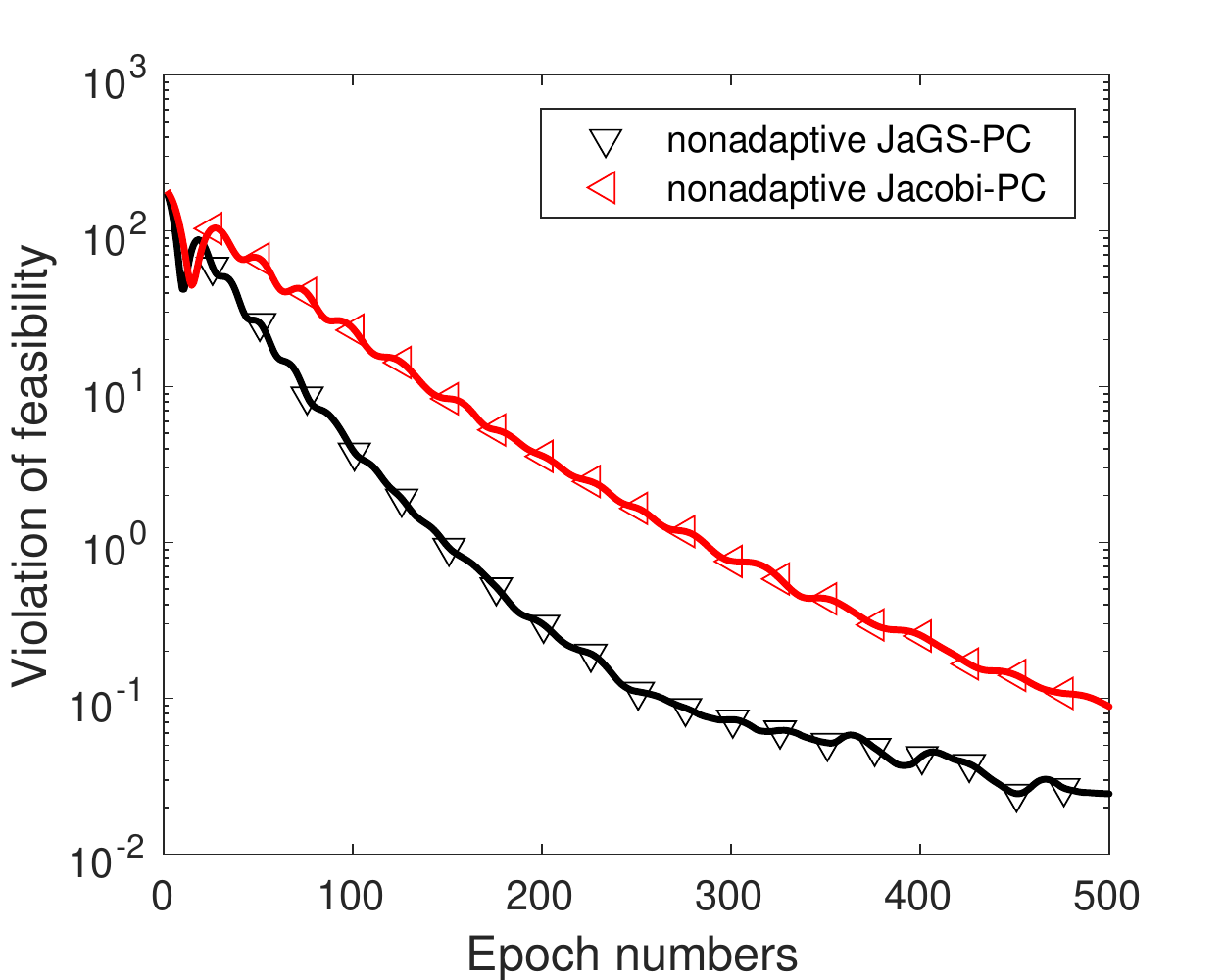}
\end{tabular}
\end{center}
\caption{Results by applying Jacobi-PC and JaGS-PC without adapting proximal terms on the quadratic programming \eqref{eq:qp}. Left: the distance of the objective to optimal value $|F(x^k)-F(x^*)|$; Right: the violation of feasibility $\|Ax^k-b\|$.}\label{fig:nonadp-qp}
\end{figure}

\begin{figure}[h]
\begin{center}
\begin{tabular}{cc}
\includegraphics[width=0.35\textwidth]{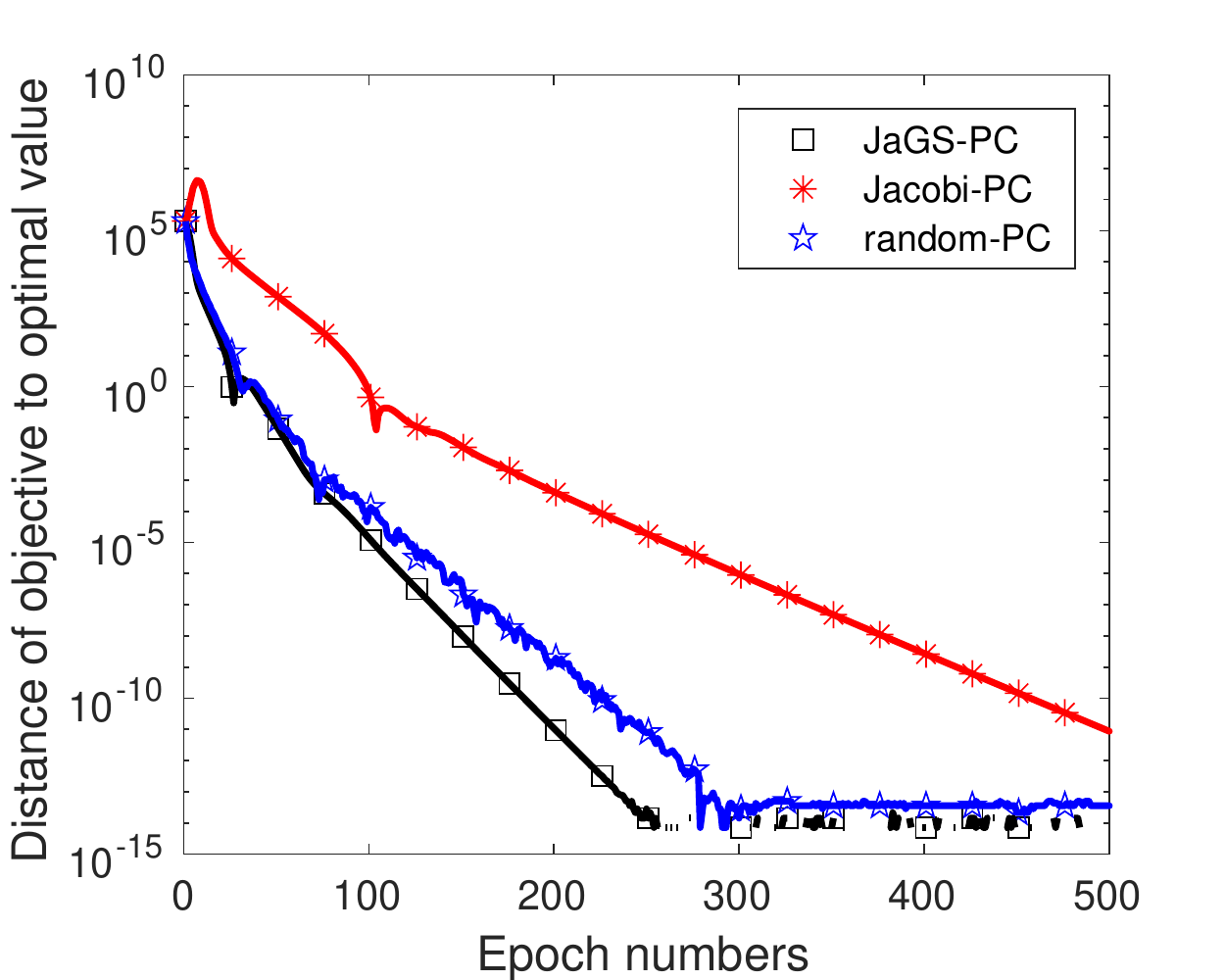} &
\includegraphics[width=0.35\textwidth]{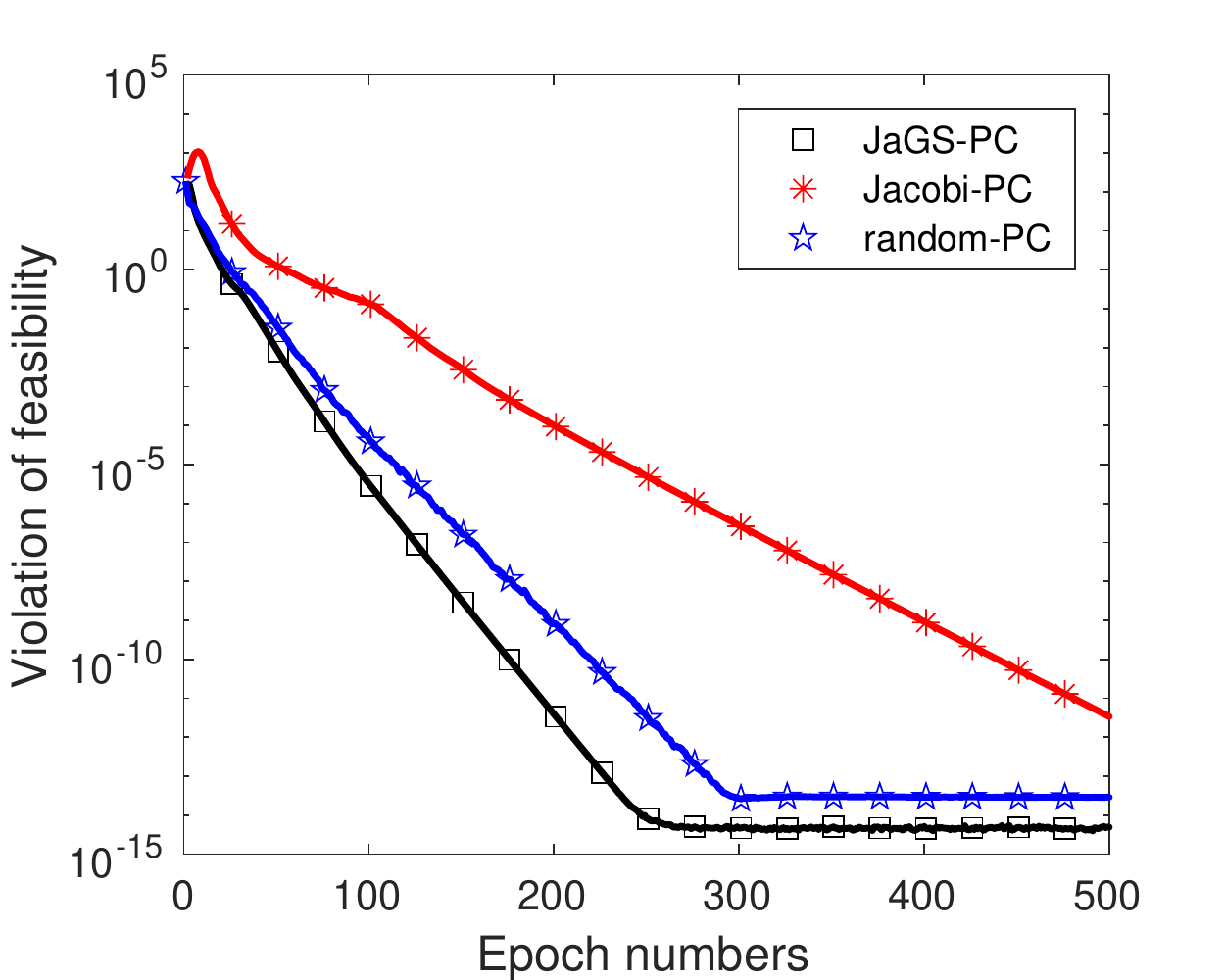}\\
\includegraphics[width=0.35\textwidth]{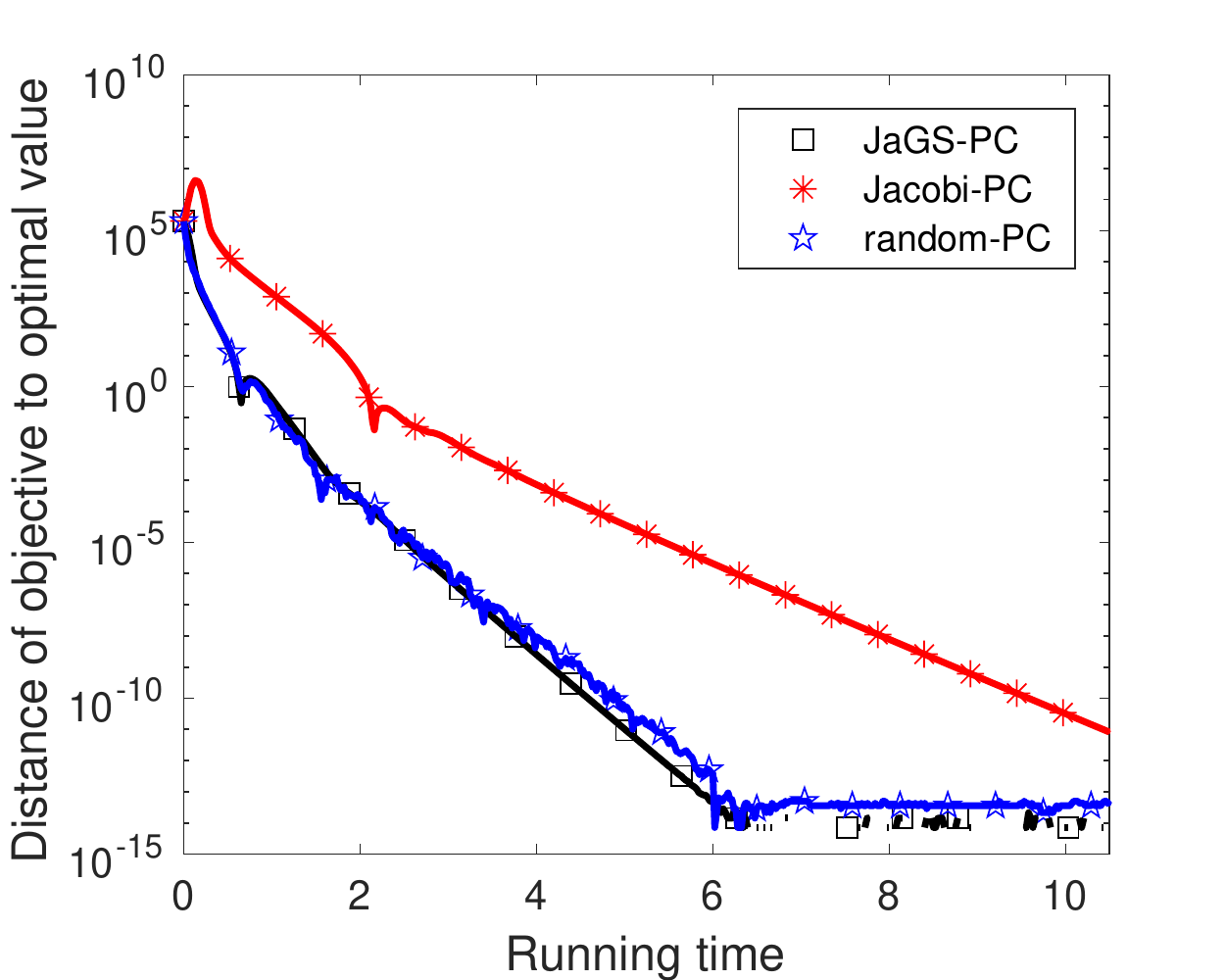} &
\includegraphics[width=0.35\textwidth]{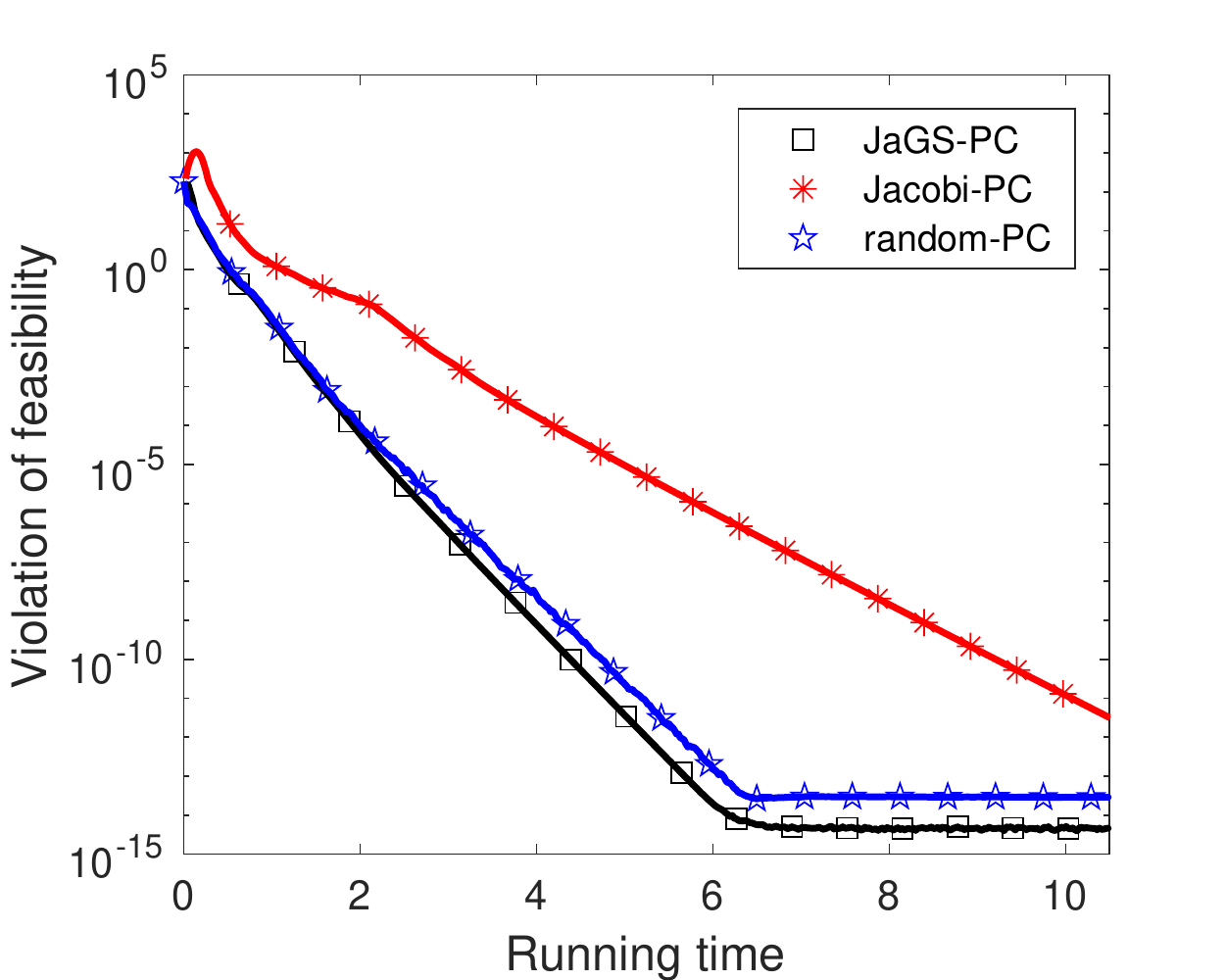}
\end{tabular}
\end{center}
\caption{Results by applying three different proximal block coordinate update methods: Jacobi-PC, JaGS-PC, and random-PC on the quadratic programming \eqref{eq:qp}. Left: the distance of the objective to optimal value $|F(x^k)-F(x^*)|$; Right: the violation of feasibility $\|Ax^k-b\|$. The running time of JaGS-PC includes that for finding the mixing matrix $W$.}\label{fig:qp}
\end{figure}

\subsection{Compressive principal component pursuit}
In this subsection, we test Jacobi-PC and JaGS-PC on
\begin{equation}\label{eq:cpcp}
\min_{X, Y} F(X,Y)=\mu\|X\|_1 + \|Y\|_*, \st \cA(X+Y) = b,
\end{equation}
where $\|X\|_1=\sum_{i,j}|X_{ij}|$, $\|Y\|_*$ denotes the matrix operator norm and equals the largest singular value of $Y$, $\cA$ is a linear operator, and $b$ contains the measurements. If $\cA$ is the identity operator, \eqref{eq:cpcp} is called the principal component pursuit (PCP) proposed in \cite{candes2011robust-PCA}, and it is compressive PCP \cite{wright2013compress-PCP} when $\cA$ is an underdetermined measuring operator. We consider the sampling operator, i.e., $\cA=\cP_\Omega$, where $\Omega$ is an index set and $\cP_\Omega$ is a projection keeping the entries in $\Omega$ and zeroing out all others.

Assume $M$ to be the underlying matrix and $b=\cA(M)$. Upon solving \eqref{eq:cpcp}, $X+Y$ recovers $M$ with sparse part $X$ and low-rank part $Y$. We use the Escalator video dataset\footnote{Available from \url{http://pages.cs.wisc.edu/~jiaxu/projects/gosus/supplement/}}, which has 200 frames of $130\times 160$ images. Each frame of 2D image is reshaped into a column vector to form a slim and tall matrix $M$ with 200 columns. For this data, $X$ will encode the foreground and $Y$ the background of the video. We generate the index set $\Omega$ uniformly at random,  and $30\%$ samples are selected from each frame of image. 


Introducing another variable $Z$, we write \eqref{eq:cpcp} equivalently to
\begin{equation}\label{eq:cpcp2}
\min_{X, Y, Z} \mu\|X\|_1 + \|Y\|_*, \st X+Y=Z,\ \cA(Z) = b,
\end{equation}
which naturally has three block variables.
Applying Algorithm \ref{alg:hyb-admm} to \eqref{eq:cpcp2} with adaptive $P_i$'s as in \eqref{eq:adap-P} and ${D}=0$ and noting $\|I+\cA^\top\cA\|_2=2$, we iteratively perform the updates:
\begin{subequations}\label{eq:hadmm-pcp}
\begin{align}
&X^{k+1}=\argmin_X \mu\|X\|_1 -\langle \Lambda^k, X\rangle +\frac{\beta}{2}\|X+Y^k-Z^k\|_F^2+\frac{d^k\beta}{2}\|X-X^k\|^2,\label{eq:pcp-x}\\
&Y^{k+1}=\argmin_Y \|Y\|_* -\langle \Lambda^k, Y\rangle +\frac{\beta}{2}\|\tilde{X}^k+Y-Z^k\|_F^2+\frac{d^k\beta}{2}\|Y-Y^k\|^2,\label{eq:pcp-y}\\
&Z^{k+1}=\argmin_Z \langle \Lambda^k-\cA^\top(\Pi^k), Z\rangle+\frac{\beta}{2}\|\tilde{X}^k+\tilde{Y}^k-Z^k\|_F^2+\frac{\beta}{2}\|\cA(Z)-b\|^2+{d^k\beta}\|Z-Z^k\|^2,\label{eq:pcp-z}\\
&\Lambda^{k+1}=\Lambda^k-\rho(X^{k+1}+Y^{k+1}-Z^{k+1}),\label{eq:pcp-lam}\\
&\Pi^{k+1}=\Pi^k-\rho(\cA(Z^{k+1})-b).\label{eq:pcp-pi}
\end{align}
\end{subequations}
Since $\cA\cA^\top=I$, all three primal subproblems have closed-form solutions. We set $\beta=\rho=0.05$ and $d_{\mathrm{inc}}=0.01$ for both JaGS-PC and Jacobi-PC methods, and $d^1=0$ for JaGS-PC and $d^1=1$ for Jacobi-PC because the latter can deviate from optimality very far away in the beginning if it starts with a small $d^1$ (see Figure \ref{fig:pcp-Jac}). They are compared to random-PC, direct ADMM, and also ADMM-GBS. Every iteration, random-PC performs one update among \eqref{eq:pcp-x} through \eqref{eq:pcp-z} with $d^k=0$ and then updates $\Lambda$ and $\Pi$ by \eqref{eq:pcp-lam} and \eqref{eq:pcp-pi} with $\rho=\frac{\beta}{3}$;  the direct ADMM sets $d^k=0,\,\forall k$ in \eqref{eq:hadmm-pcp}; ADMM-GBS runs the direct ADMM first and then performs a correction step by Gauss back substitution. We use the same $\beta$ and $\rho$ for the direct ADMM and ADMM-GBS and set the correction step parameter of ADMM-GBS to 0.99. On solving the SDP \eqref{eq:u-sdp}, we have for JaGS-PC $d_{\max}=0.4270$ and the mixing matrix:
$$W=\left[\begin{array}{ccc}1 &   1 &   1\\
    0.3691      &   1 &   1\\
   -0.2618   & 0.3691 &        1\end{array}\right].$$Figure \ref{fig:pcp} plots the results by all five methods, where the optimal solution is obtained by running JaGS-PC to 10,000 epochs. From the figure, we see that JaGS-PC performs significantly better than Jacobi-PC. JaGS-PC, direct ADMM and ADMM-GBS perform almost the same, and random-PC is the worst.  
Note that although direct ADMM works well on this example, its convergence is not guaranteed in general. 

\begin{figure}
\begin{center}
\begin{tabular}{cc}
\includegraphics[width=0.35\textwidth]{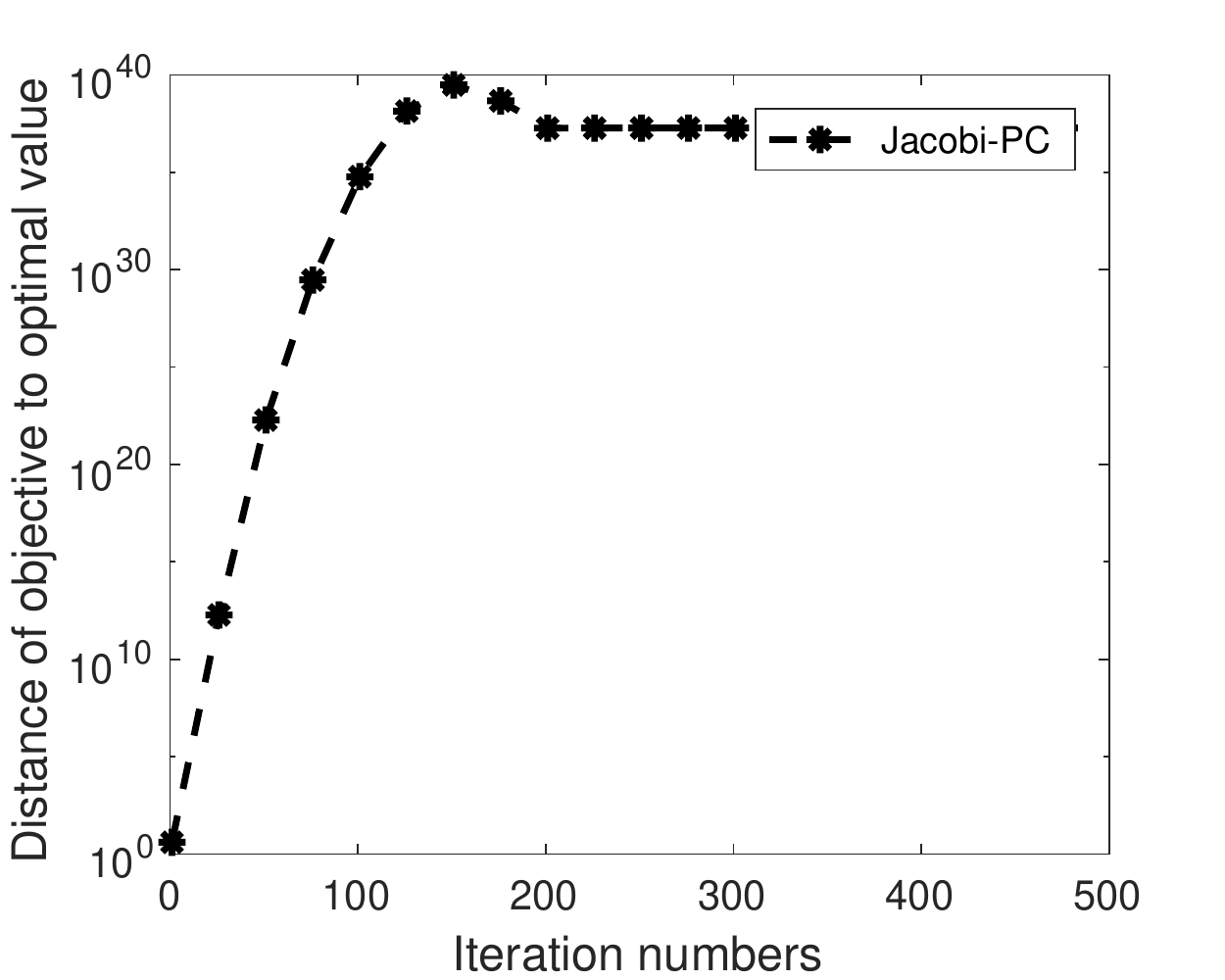} &
\includegraphics[width=0.35\textwidth]{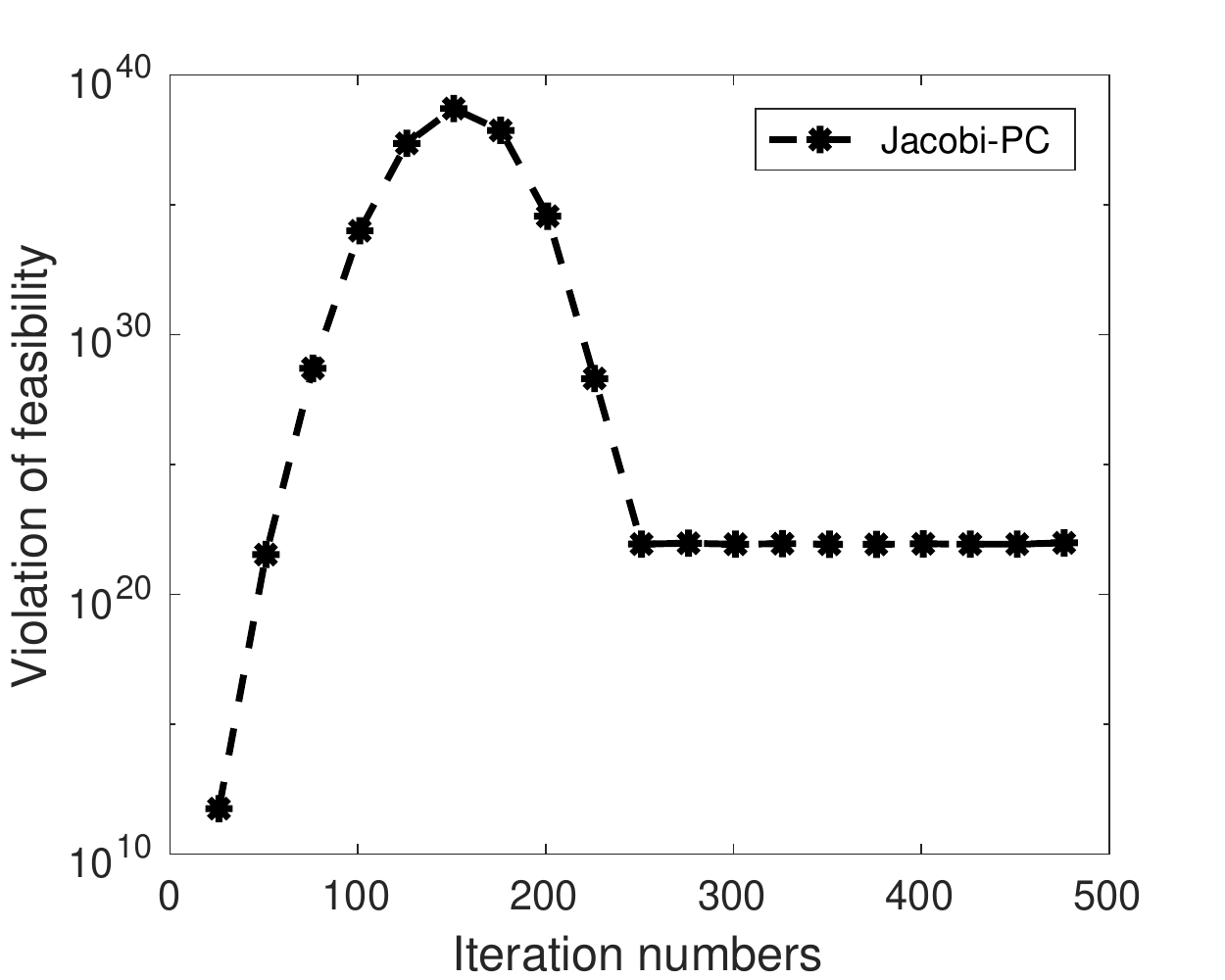}
\end{tabular}
\end{center}
\caption{Results by Jacobi-PC with $d^1=0$ and $d_{\mathrm{inc}}=0.01$ for solving the compressive principal component pursuit \eqref{eq:cpcp2} on the Escalator dataset. Left: the relative error between the objective and optimal value $\frac{|F(X^k,Y^k)-F(X^*,Y^*)|}{F(X^*,Y^*)}$; Right: relative violation of feasibility: $\frac{\|X^k+Y^k-Z^k\|_F+\|\cA(Z^k)-b\|_F}{\|M\|_F}$.}\label{fig:pcp-Jac}
\end{figure}

\begin{figure}[h]
\begin{center}
\begin{tabular}{cc}
\includegraphics[width=0.35\textwidth]{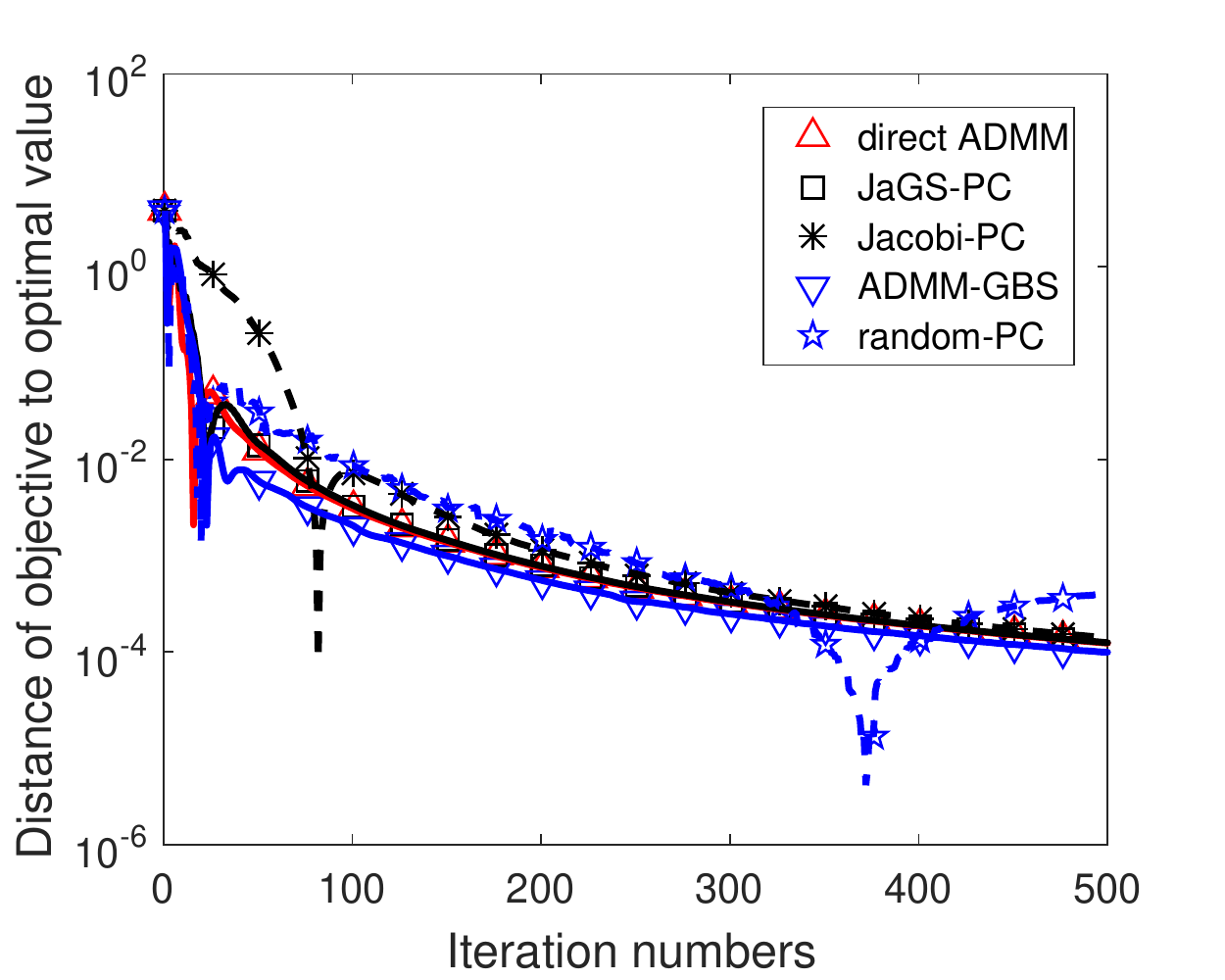} &
\includegraphics[width=0.35\textwidth]{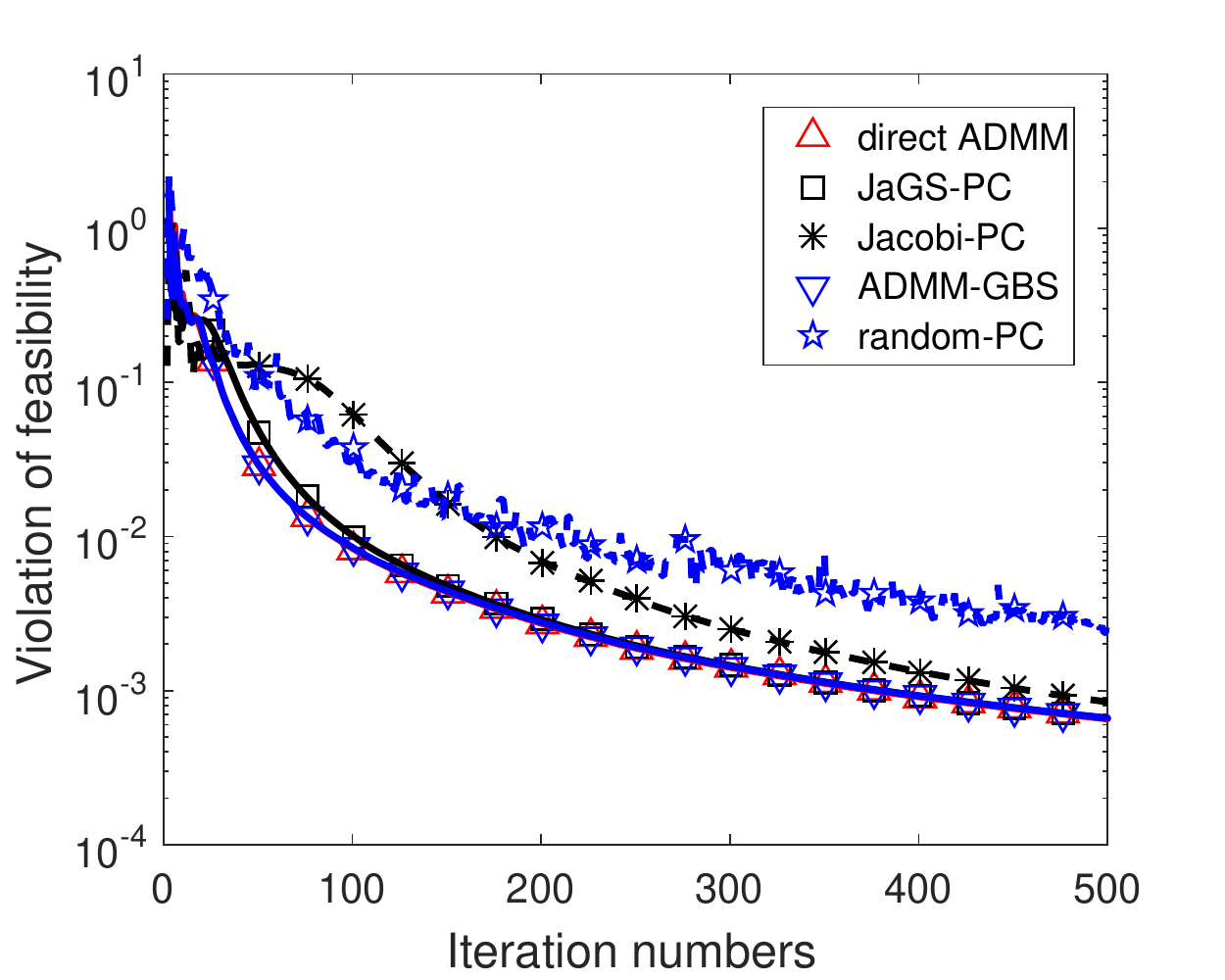}\\
\includegraphics[width=0.35\textwidth]{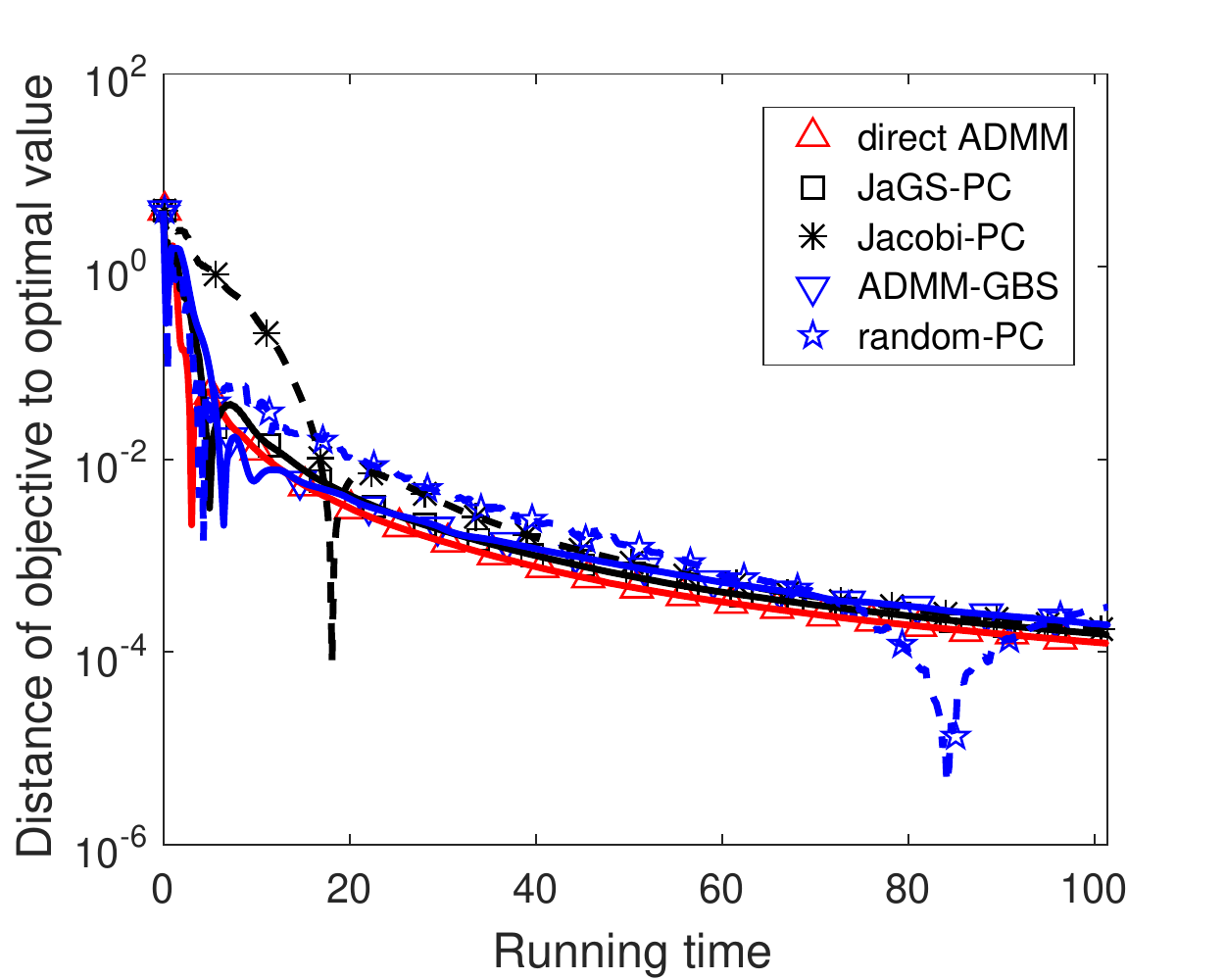} &
\includegraphics[width=0.35\textwidth]{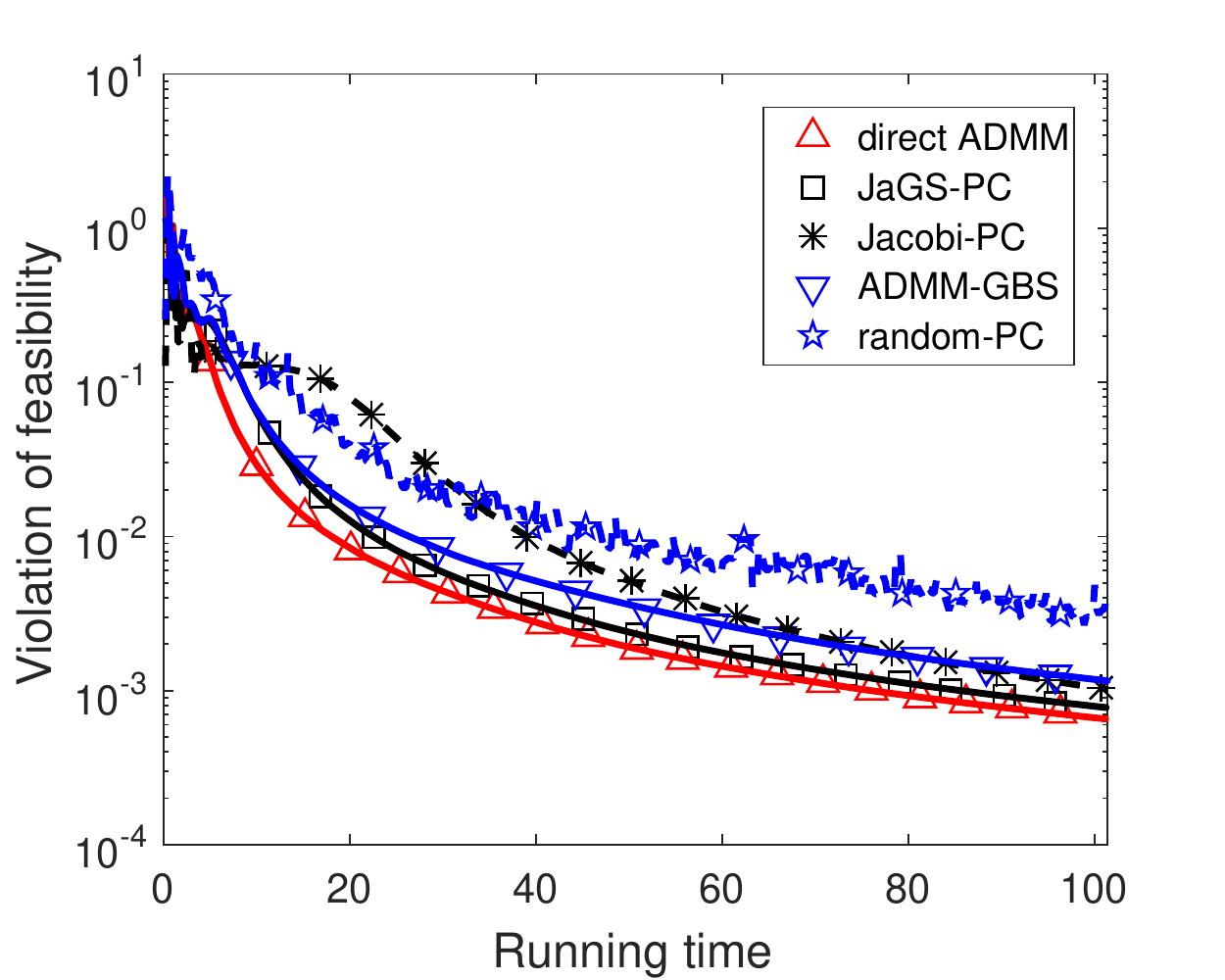}
\end{tabular}
\end{center}
\caption{Results by five different methods for solving the compressive principal component pursuit \eqref{eq:cpcp2} on the Escalator dataset. Left: the relative error between the objective and optimal value $\frac{|F(X^k,Y^k)-F(X^*,Y^*)|}{F(X^*,Y^*)}$; Right: relative violation of feasibility: $\frac{\|X^k+Y^k-Z^k\|_F+\|\cA(Z^k)-b\|_F}{\|M\|_F}$. The running time of JaGS-PC includes that for finding the mixing matrix $W$.}\label{fig:pcp}
\end{figure}

\subsection{Multi-class support vector machine}
In this subsection, we test Jacobi-PC, JaGS-PC, and random-PC on the multi-class support vector machine (MSVM) problem that is considered in \cite{xu2015admm-msvm}:
\begin{equation}\label{eq:msvm}
\min_X F(X)=\frac{1}{n}\sum_{i=1}^n\sum_{j=1}^c \iota_{\cdot\neq j}(b_i)[x_j^\top a_i+1]_+ +\mu\|X\|_1,\st Xe = 0,
\end{equation}
where $x_j$ is the $j$-th column of $X$, $\{(a_i,b_i)\}_{i=1}^n$ is the training dataset with label $b_i\in\{1,\ldots,c\},\,\forall i$, $\iota_{\cdot\neq j}(b_i)$ equals \emph{one} if $b_i\neq j$ and \emph{zero} otherwise, and $[d]_+=\max(0,d)$. We set the number of classes to $c=3$ and randomly generate the data according to Gaussian distribution $\cN(v_j,\Sigma_j)$ for the $j$-th class, where $v_j\in\RR^p$ and $\Sigma_j\in\RR^{p\times p}$ for $j=1,2,3$ are
\begin{align*}
&v_1=\left[\begin{array}{c}E_{s\times 1}\\ 0\end{array}\right],\, v_2=\left[\begin{array}{c}0_{s/2 \times  1}\\E_{s\times  1}\\ 0\end{array}\right],\,v_3=\left[\begin{array}{c}0_{s\times 1}\\ E_{s\times 1}\\ 0\end{array}\right],\,\Sigma_1=\left[\begin{array}{cc}\sigma E_{s\times s}+(1-\sigma) I & 0\\ 0 & I\end{array}\right],\\[0.1cm]
& \Sigma_2= \left[\begin{array}{ccc}I_{\frac{s}{2}\times\frac{s}{2}} & 0 & 0\\ 0 &\sigma E_{s\times s}+(1-\sigma) I & 0\\ 0 & 0 & I\end{array}\right],\,\Sigma_3= \left[\begin{array}{ccc}I_{s\times s} & 0 & 0\\ 0 &\sigma E_{s\times s}+(1-\sigma) I & 0\\ 0 & 0 & I\end{array}\right],
\end{align*}
where $E, I, 0$  respectively represent all-one, identity, and all-zero matrices of appropriate sizes, and the subscript specifies the size. The parameter $\sigma$ measures correlation of features. This kind of dataset has also been used in \cite{xu2015HHSVM} for testing binary SVM. In the test, we set $\sigma=0.1, \mu=0.001, p=200$ and $n= 300$, each class consisting of 100 samples.

Letting $A=[a_1,\ldots, a_n]$ and $Y=A^\top X+1$, we write \eqref{eq:msvm} equivalently to
\begin{equation}\label{eq:msvm2}
\min_{X,Y} \frac{1}{n}\sum_{i=1}^n\sum_{j=1}^c \iota_{\cdot\neq j}(b_i)[Y_{ij}]_+ +\mu\|X\|_1,\st A^\top X- Y + 1 = 0,\, Xe = 0.
\end{equation}
To apply Algorithm \ref{alg:hyb-admm} to the above model, we partition the variable into four blocks $(x_1, x_2, x_3, Y)$. Linearization to the augmented term is employed, i.e., ${D}=I$ in \eqref{eq:adap-P}. The parameters are set to $\beta=\rho=0.005$ and $d_{\mathrm{inc}}=0.1$ for both Jacobi-PC and JaGS-PC, and $d^1=0.5$ for JaGS-PC and $d^1=1$ for Jacobi-PC because again the latter can deviate from optimalty far away in the beginning if it starts with a small $d^1$ (see Figure \ref{fig:svm-Jac}). Each iteration, random-PC picks one block from $x_1,x_2,x_3$ and $Y$ uniformly at random and updates it by minimizing the proximal linearized augmented Lagrangian function with respect to the selected block and the other three blocks fixed. The proximal parameter is adaptively increased as well. In the multiplier update, $\rho=\frac{\beta}{4}$ is set, and $\beta$ is the same as that for JaGS-PC. On solving the SDP \eqref{eq:u-sdp}, we have for JaGS-PC $d_{\max}=1.8711$ and the mixing matrix:
$$W=\left[\begin{array}{cccc}1  &  1 &   1 &   1\\
    0.5353  &  1 &    1 &    1 \\
    0.0705 &   0.5353  &   1 &    1\\
   -0.3942  &   0.0705  &  0.5353  &  1\end{array}\right].$$We plot the results in Figure \ref{fig:svm}, where the optimal solution is given by CVX \cite{grant2008cvx} with ``high precision'' option. In terms of objective value, JaGS-PC and random-PC perform significantly better than Jacobi-PC, and the former two are comparably well. However, random-PC is significantly worse than JaGS-PC and Jacobi-PC in terms of feasibility violation. 

\begin{figure}
\begin{center}
\begin{tabular}{cc}
\includegraphics[width=0.35\textwidth]{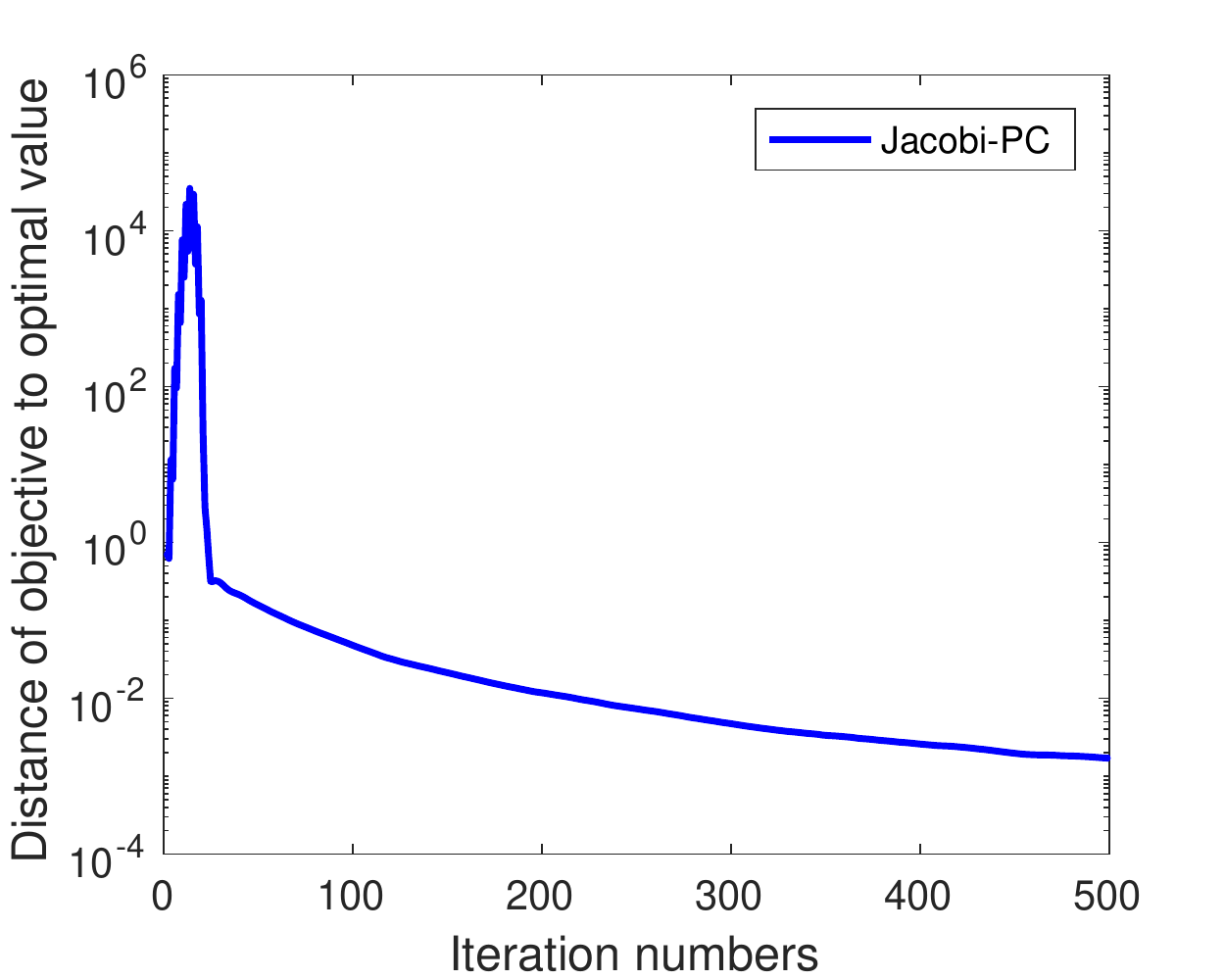} &
\includegraphics[width=0.35\textwidth]{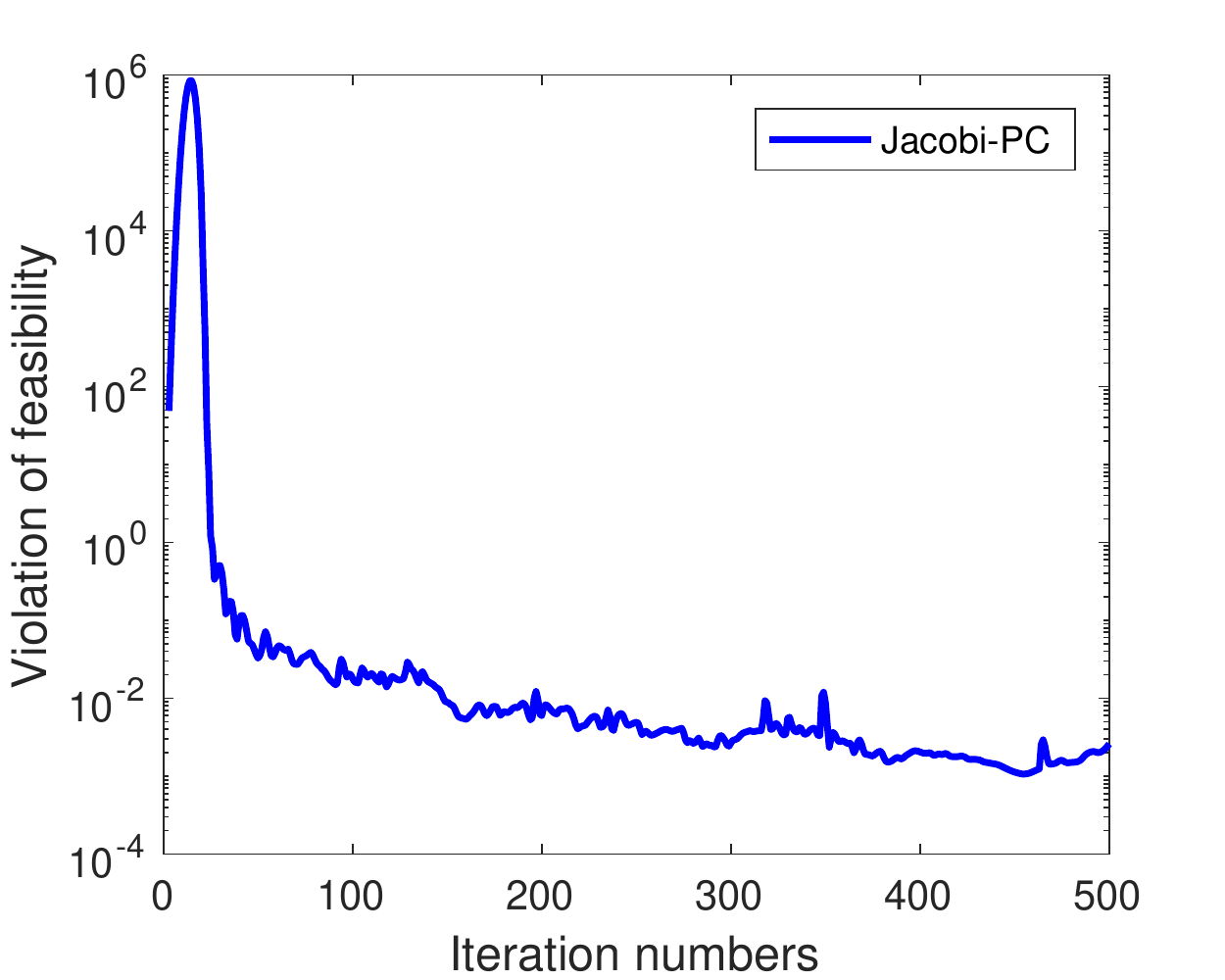}
\end{tabular}
\end{center}
\caption{Results by Jacobi-PC with $d^1=0.5$ and $d_{\mathrm{inc}}=0.1$ for solving the multiclass support vector machine problem \eqref{eq:msvm} on randomly generated data. Left: distance of the objective to optimal value $|F(X^k)-F(X^*)|$; Right: violation of feasibility $\|X^k e\|$.}\label{fig:svm-Jac}
\end{figure}

\begin{figure}[h]
\begin{center}
\begin{tabular}{cc}
\includegraphics[width=0.35\textwidth]{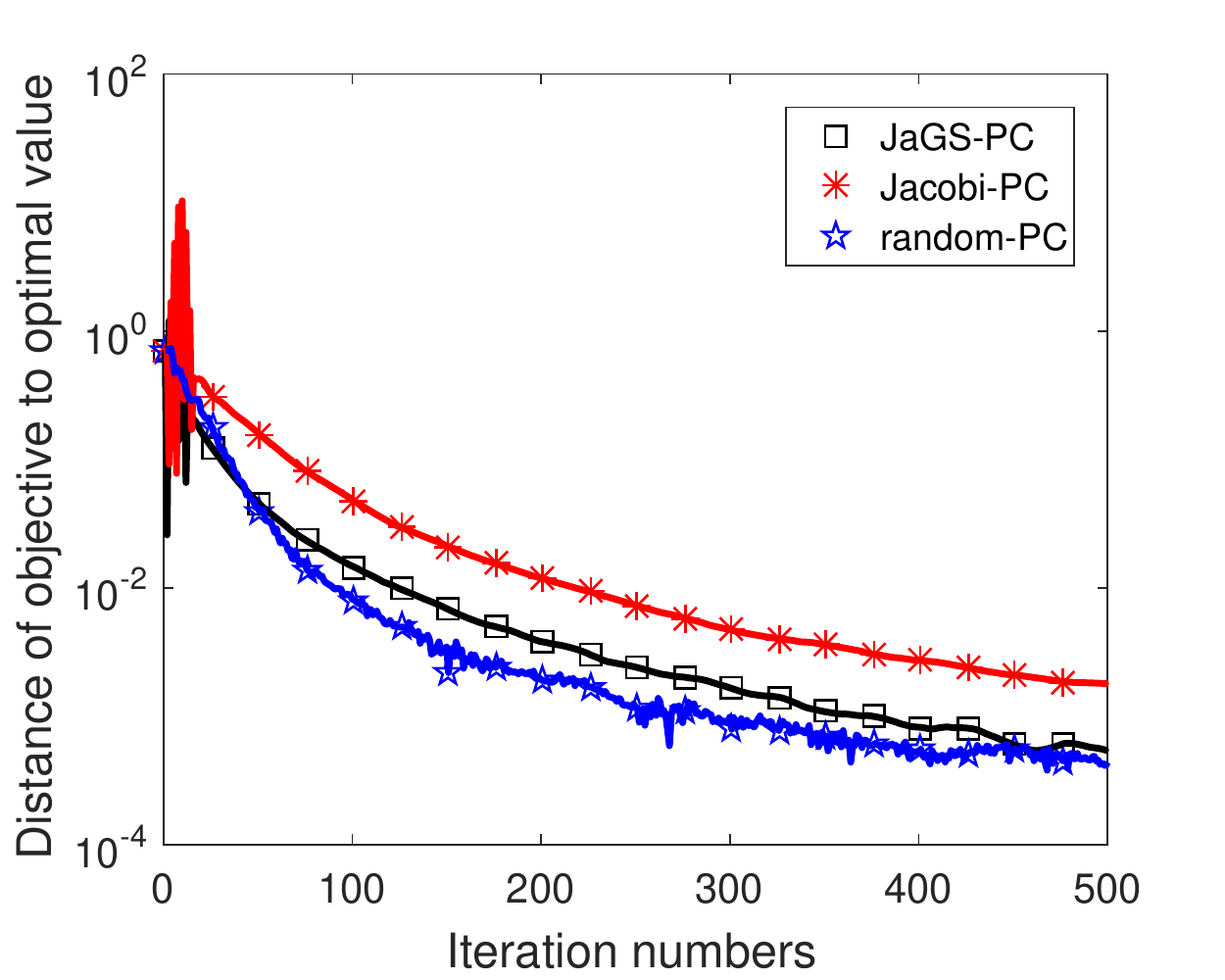} &
\includegraphics[width=0.35\textwidth]{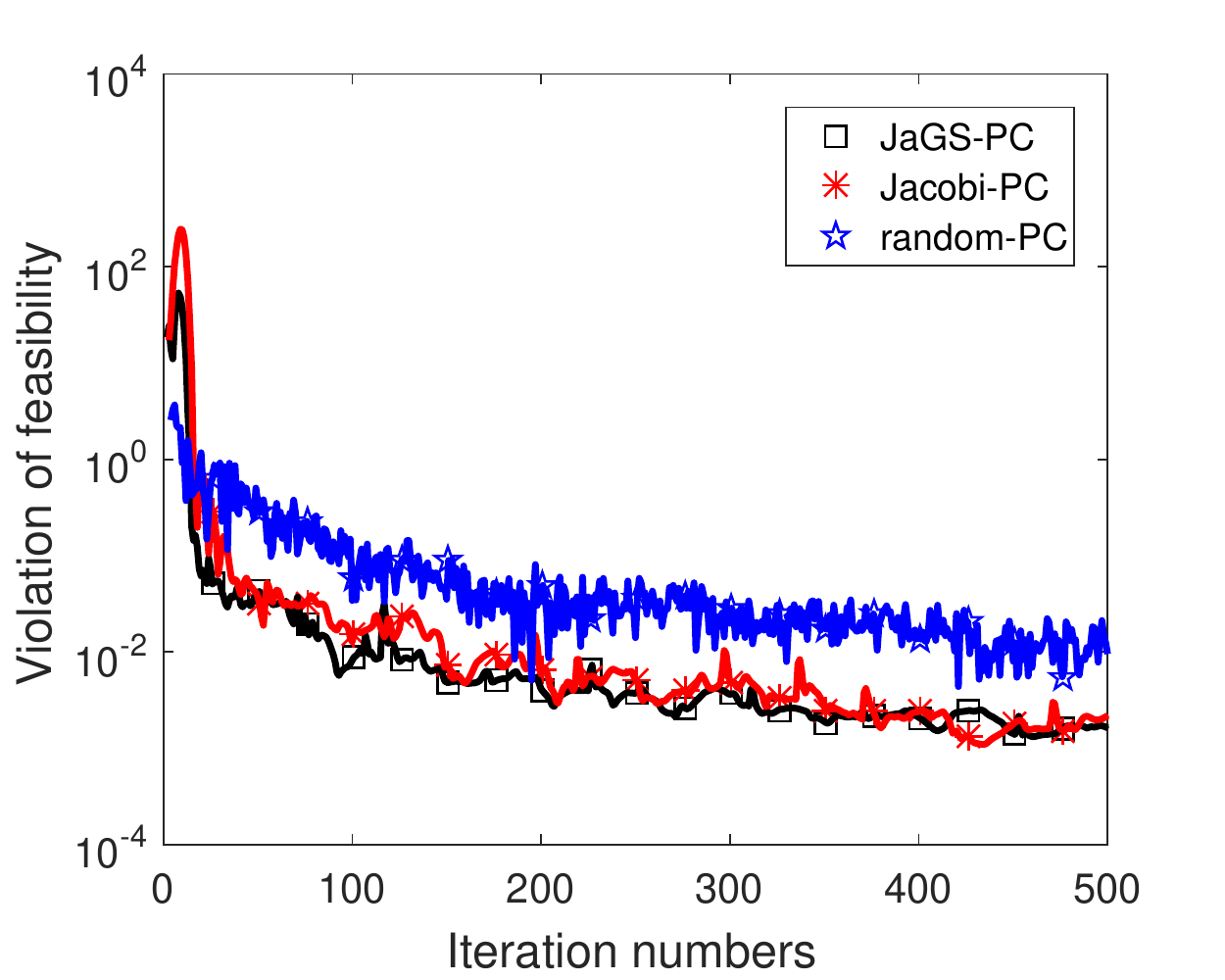}\\
\includegraphics[width=0.35\textwidth]{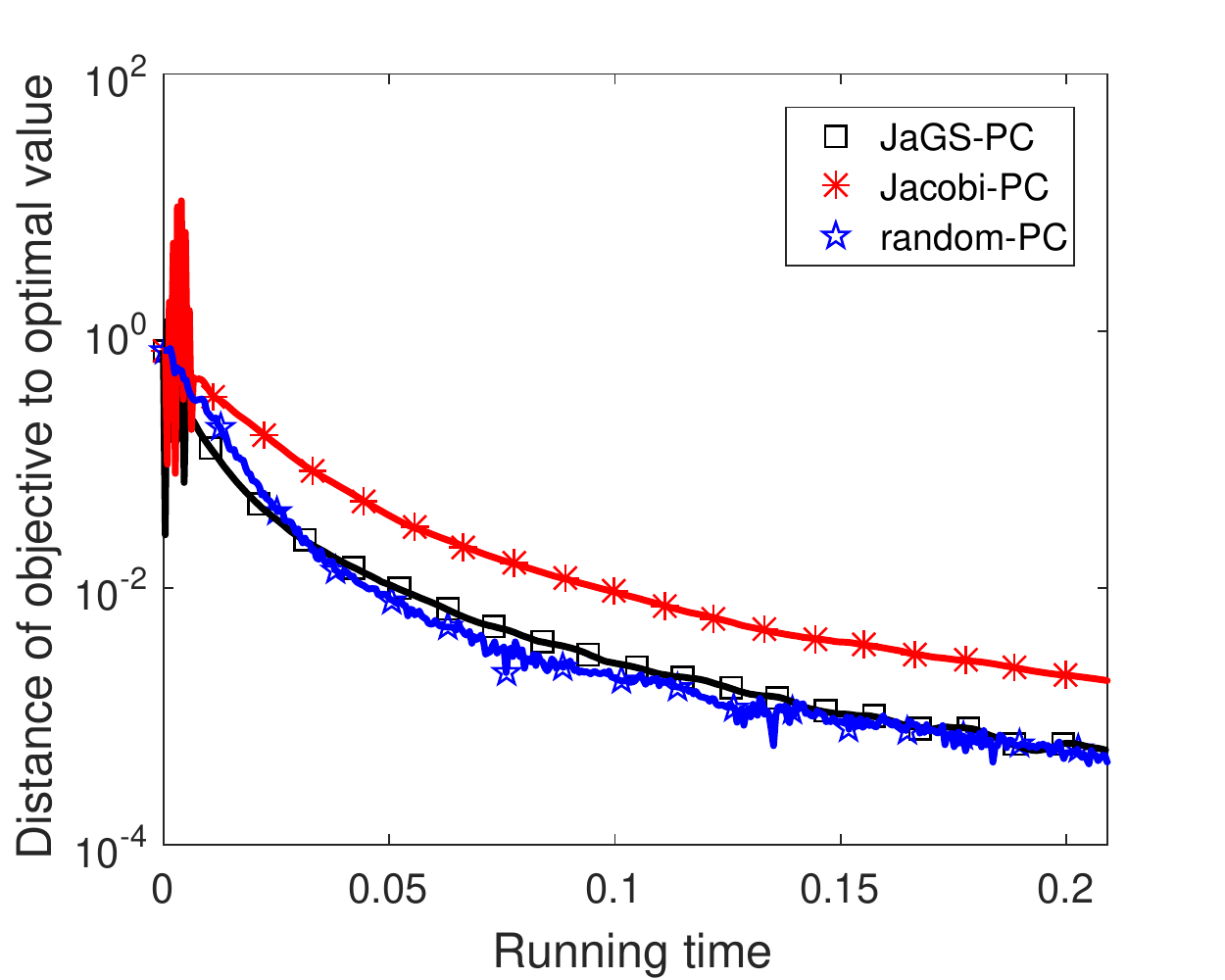} &
\includegraphics[width=0.35\textwidth]{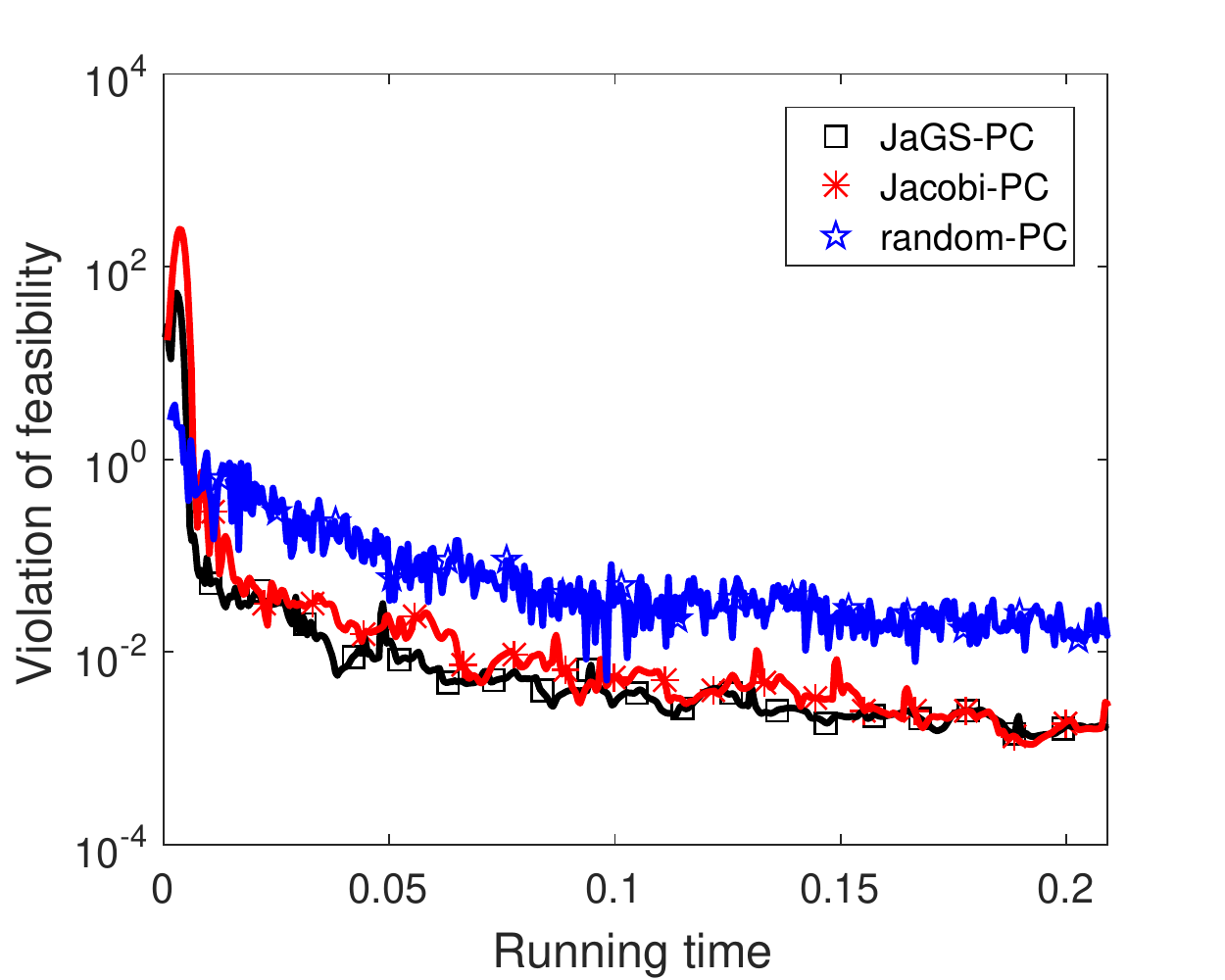}
\end{tabular}
\end{center}
\caption{Results by Jacobi-PC and JaGS-PC for solving the multiclass support vector machine problem \eqref{eq:msvm} on randomly generated data. Left: distance of the objective to optimal value $|F(X^k)-F(X^*)|$; Right: violation of feasibility $\|X^k e\|$. The running time of JaGS-PC includes that for finding the mixing matrix $W$.}\label{fig:svm}
\end{figure}

\section{Conclusions}\label{sec:conclusion}
We have proposed a hybrid Jacobian and Gauss-Seidel block coordinate update method for solving linearly constrained convex programming. The method performs each primal block variable update by minimizing a function that approximates the augmented Lagrangian at affinely combined points of the previous two iterates. We have presented a way to choose the mixing matrix with desired properties. Global iterate sequence convergence and also sublinear rate results of the hybrid method have been established. In addition, numerical experiments have been performed to demonstrate its efficiency. 

\section*{Acknowledgements} This work is partly supported by NSF grant DMS-1719549. The author would like to thank two anonymous referees for their careful review and constructive comments, which help greatly improve the paper.

\bibliographystyle{abbrv}
\bibliography{alm,bcd,mblk-adm}

\end{document}